\documentclass[12pt]{article}
\usepackage{amsmath,amssymb}
\usepackage{amsthm}
\usepackage{graphicx,tikz}
\usepackage{enumerate}
\numberwithin{equation}{section}
\newtheorem{Thm}{Theorem}[section]
\newtheorem{Prop}[Thm]{Proposition}
\newtheorem{Lem}[Thm]{Lemma}
\newtheorem{Conj}[Thm]{Conjecture}
\newtheorem{Cor}[Thm]{Corollary}
\newtheorem{Fact}[Thm]{Fact}

\theoremstyle{remark}

\newtheorem{Exa}[Thm]{Example}
\newtheorem{Rem}[Thm]{Remark}

\theoremstyle{definition}

\newtheorem{Def}[Thm]{Definition}
\newtheorem{Not}[Thm]{Notation}
\newtheorem*{Def*}{Definition}

\numberwithin{equation}{section}

\newcommand{\g}[1]{{\mathfrak #1}}
\newcommand{\m}[1]{\mathbb{ #1}}
\newcommand{\mc}[1]{\mathcal{ #1}}

\def\al{\alpha}       \def\be{\beta}        \def\ga{\gamma}
\def\de{\delta}       \def\eps{\varepsilon}  
\def\th{\theta}       %\def\im{\imath}
       \def\la{\lambda}      
\def\si{\sigma}                
\def\ph{\varphi}               
\def\om{\omega}       \def\Ga{\Gamma}       \def\De{\Delta}
\def\La{\Lambda}      \def\Si{\Sigma}       \def\Ph{\Phi}
                  \def\Om{\Omega}

\theoremstyle{definition}

\theoremstyle{remark}

\newtheorem{Rmq}[Thm]{Remark}

\numberwithin{equation}{section}
\newfont{\goth}{eufm10 at 12pt}
\newfont{\gots}{eufm8 at 9pt}

\def\bt{\begin{Thm}}
\def\et{\end{Thm}}
\def\br{\begin{Rmq}}
\def\er{\end{Rmq}}

\def\bc{\begin{Cor}}
\def\ec{\end{Cor}}
\def\bp{\begin{Prop}}
\def\ep{\end{Prop}}
\def\bl{\begin{Lem}}
\def\el{\end{Lem}}
\def\bd{\begin{Def}}
\def\ed{\end{Def}}
\def\bq{\begin{quotation}}
\def\eq{\end{quotation}}
\def\bfa{\begin{Fact}}
\def\efa{\end{Fact}}

\def\ra{\rightarrow}
\def\vs{\vspace{1em}}

\begin{document}
\title{
Arithmeticity of discrete subgroups\\ containing horospherical lattices
}
\author{Yves Benoist and S\'ebastien Miquel
}
\date{}
\maketitle
\vspace{-2em}
%\centerline{\footnotesize 
%\mbox{Preliminary version } }

\begin{abstract}{
\noindent Let $G$ be a semisimple real algebraic Lie group of real rank at least two   and 
$U$ be the unipotent radical of a non-trivial parabolic subgroup. 
We prove that a discrete Zariski dense subgroup of $G$ 
that contains an irreducible lattice of $U$
is an arithmetic lattice of $G$. 
This solves a conjecture of Margulis and extends previous work of Hee Oh.}\end{abstract}
{\footnotesize \tableofcontents}

%1
\section{Introduction}
\label{secintrod}

In the mid $90's$, G. Margulis raised the following
conjecture.

\begin{Conj}
\label{conmar}
Let $G$ be a semisimple real algebraic Lie group 
of real rank at least $2$ and $U$ 
be a non-trivial horospherical subgroup of $G$.
Let $\Gamma$ be a discrete Zariski dense subgroup of $G$
that contains an irreducible lattice $\Delta$  of $U$. Then $\Gamma$
is a non-cocompact irreducible arithmetic lattice of $G$.
\end{Conj}

This conjecture is first quoted by Hee Oh in \cite{Oh96} and in \cite{Oh00}. Indeed her work  in \cite{Oh98a} is a major contribution towards it  (see Section \ref{secmotpre}). 
This conjecture is also reported by Lizhen Ji in  his summary \cite[Section 7]{Ji08} 
of the work of Margulis. The main feature here is that 
``$\Ga$ is a lattice'' is a conclusion, not an assumption.
The aim of this paper is to prove (Theorem \ref{thmmar}) that
Conjecture \ref{conmar} is true.

%11
\subsection{Arithmeticity of discrete groups}
\label{secmaires}

\bq
We quickly recall the relevant definitions.  
See Chapter~\ref{secnotati} for more details
and Section \ref{secexahor} for examples.
\eq

A {\it semisimple real algebraic Lie group} is  a  subgroup 
$G\subset {\rm GL}(N,\m R)$ which is defined by polynomial equations with real coefficients
and whose Lie algebra $\g g$ is semisimple. The real rank of $G$ is the dimension of the 
maximal split tori of $G$. 
The group $G$ is endowed with two topologies:
the Lie group topology and the Zariski topology.

A  {\it horospherical subgroup} $U$ of $G$ 
is the {\it stable group} of an element $g$ in $G$ 
i.e. 
$\displaystyle U:=\{u\in G\mid \lim_{n\ra\infty} g^n ug^{-n}=e\}$. The normalizer $P$ of such a group $U$ is called  a {\it parabolic subgroup}, and the group $U$ is the unipotent radical of $P$.

A {\it lattice} $\Delta$ of $U$ is a discrete cocompact subgroup of $U$.
It is said to be 
{\it irreducible}
if it intersects trivially 
every proper normal subgroup of $G$
(see also Remark \ref{remirrhor} for equivalent definitions).
If such a lattice $\Delta$ exists, the group $U$ 
meets non-trivially all the simple normal subgroups of $G$. 
In particular, the group $G$ has no compact factors.

A discrete subgroup $\Gamma$ of $G$ is said to be 
{\it irreducible} if, for
all proper normal subgroup $G'$ of $G$
the intersection $\Ga\cap G'$ is finite.
A {\it lattice}  $\Gamma$ of $G$ is a discrete subgroup  of $G$ 
such that the  quotient $G/\Gamma$ has finite volume.
For instance, discrete cocompacts subgroups of $G$ are lattices of $G$.
In this preprint we will only deal with irreducible non-cocompact lattices of $G$.

If one can choose an embedding 
$G\hookrightarrow {\rm GL}(N,\m R)$
and   polynomial equations defining $G$ with rational coefficients, the group 
$G_{\m Q}:= G\cap {\rm GL}(N,\m Q)$
is dense in $G$. This group $G_{\m Q}$ is then called a $\m Q$-form of $G$ and the subgroup of integral points $G_{\m Z}:= G\cap {\rm GL}(N,\m Z)$ is a lattice of $G$. 
When $G$ is adjoint, an irreducible non-cocompact lattice $\Gamma$ of $G$ is said to be {\it arithmetic} if there exists a $\m Q$-form of $G$ 
such that $\Gamma$ and $G_{\m Z}$
are commensurable.

\begin{Rem}
- When $G$ is simple, every lattice of $U$ is irreducible.\\ 
- The irreducibility assumption on the lattice $\Delta$ of $U$ in Conjecture \ref{conmar}
can be weakened to an {\it indecomposability} assumption
(see Definition \ref{defirrlat}).\\
- A non-cocompact irreducible arithmetic lattice $\Gamma$ of $G$ is always 
Zariski dense and always intersects cocompactly
at least one non-trivial horospherical subgroup $U$ of $G$.\\ 
- The higher rank assumption in Conjecture \ref{conmar} is necessary:  
In $G={\rm SL}(2,\m R)$, the subgroup of ${\rm SL}(2,\m Z)$
generated by the matrices 
\mbox{\scriptsize$
\left(\!\!
\begin{array}{cc}
1&3\\
0&1
\end{array}\!\!
\right) 
$} 
and 
\mbox{\scriptsize$
\left(\!\!
\begin{array}{cc}
1&0\\
3&1
\end{array}\!\!
\right) 
$} 
has infinite index in ${\rm SL}(2,\m Z)$ and therefore
is not a lattice in $G$.
More generally when $G$ is a simple real algebraic Lie group of real rank  one, 
when $U$, $U^-$ are two distinct non-trivial horospherical subgroups of $G$, 
and when $\Delta$, $\Delta^-$ are lattices in $U$ and $U^-$,
there always  exist finite index subgroups $\Delta'$ and  ${\Delta'}^-$ 
in $\Delta$ and $\Delta^-$ that play ping-pong on the flag variety $G/P$.
The group generated by  $\Delta'$ and  ${\Delta'}^-$ is then a discrete
subgroup of $G$, isomorphic to the free product  $\Delta' \star {\Delta'}^-$, and
of infinite covolume in $G$.\\
- When $G$ isn't adjoint, an irreducible non-cocompact lattice of $G$ is arithmetic if
its image in the adjoint group of $G$ is arithmetic.
\end{Rem}

%12
\subsection{Motivations and previous results}
\label{secmotpre}
\bq
We present some history of Conjecture \ref{conmar} and related questions, aswell as previous partial results towards it.
\eq

The fundamental finiteness theorem of Borel and Harish-Chandra in \cite{BorelHarish62}
in the $60$'s tells us that  in a real semisimple algebraic Lie group $G$ defined over $\m Q$ 
the group  $G_{\m Z}$ of integral points is a lattice. 

The celebrated arithmeticity theorem of Margulis
in \cite{Margulis84} and \cite{Margulis91}
in the $70$'s
tells us that {\it an irreducible lattice $\Gamma_0$
in a real semisimple algebraic group $G$
with no compact factors and of real rank at least two is an arithmetic group.} 
This means that there exists a real semisimple algebraic group $H$ defined over $\m Q$ 
and a group morphism 
$\pi: H\ra G$ with compact kernel and finite index image such that 
$\Ga_0$ is commensurable to $\pi(H_\m Z)$. 

The proof was a fantastic 
application of ergodic theory.
Before proving the general case,  Margulis 
first proved in \cite{Margulis74}
his arithmeticity theorem for 
non-cocompact lattices and showed that, in this case, one can choose 
$\pi$ to have finite kernel.
The strategy for this first case 
relied on a previous result of Kazhdan and Margulis
in \cite{KazhdanMargulis68},
which implies that there exists a horospherical subgroup 
$U$ of $G$ such that the intersection $U\cap \Gamma_0$ is an irreducible lattice in $U$
see \cite[Theorem 7.1.1]{Margulis74} and also \cite{Raghunathan75}. 
It also used a previous result
of Borel in \cite{Borel60} which says that 
the lattice $\Gamma_0$ is Zariski dense in $G$. 

The Margulis arithmeticity theorem was the resolution of a conjecture of Selberg.  In fact
A. Selberg, using tricky computations, 
was able to prove the arithmeticity 
of non-cocompact irreducible lattices $\Gamma_0$ 
when $G$ 
is a product of copies of ${\rm SL}(2,\m R)$ and  
${\rm SL}(2,\m C)$. 
In the early $90$'s 
A. Selberg announced that his unpublished proof applied
not only to irreducible lattices but also 
to any discrete and Zariski dense subgroup $\Gamma$ of $G$
that intersects the maximal unipotent subgroup
$U$ of $G$ in an irreducible lattice.

Moreover, by a previous result of Raghunathan and Venkataramana 
in \cite{Raghunathan76} and \cite{Venkataramana94},
it was known that Conjecture \ref{conmar} was true 
for a subgroup of an arithmetic group.  Indeed, their results say that
if $G$ is an absolutely $\m Q$-simple algebraic group
of real rank at least two,
and $U$, $U^-$ is a pair of non-trivial opposite maximal horospherical subgroups defined over $\m Q$
then any subgroup $\Gamma$ of $G_{\m Z}$ that contains lattices in
both $U$ and $U^-$ has finite index
in $G_{\m Z}$.

All these partial results lead G. Margulis to 
formulate Conjecture \ref{conmar}. 
In her PhD thesis \cite{Oh98a} in 1997, Hee Oh  proved 
Conjecture \ref{conmar} in many cases
where $G$ is absolutely simple.
For instance she proved it for all simple $\m R$-split 
groups $G$, except for $G={\rm SL}(3,\m R)$ 
and $P$ a minimal parabolic subgroup.
Her proof relied on a clever combination of  Ratner's  
classification theorem (\cite{Ratner91a}, \cite{Ratner91b}),
Weil's local rigidity theorem for lattices  and
a construction of $\m Q$-forms of Margulis
(\cite{Margulis74}). 
The remaining unproven cases with $G$ absolutely simple
consisted of four infinite families of pairs $(G,P)$ and two exceptional cases. 
Among them was the pair $G={\rm SL}(n,\m H)$, $n\geq 3$ and $P$
the stabilizer of a quaternionic line in $\m H^n$, and also the pair
$G={\rm SO}(2,n+2)$, $n\geq 2$ and $P$ the  
stabilizer of an isotropic $2$-plane.
Neither can this method  deal with   products
$G$ of real rank one groups.

Selberg's preprint \cite{Selberg08}
became available only in 2008.
After simplifying Selberg's ideas in \cite{BenoistOh10a}, 
Benoist and Oh solved in \cite{BenoistOh10b} 
the last remaining $\m R$-split simple group, i.e. 
when $G={\rm SL}(3,\m R)$ 
and $P$ is a minimal parabolic subgroup.

Motivated by these new results, the second author
focused on Conjecture \ref{conmar} in his PhD thesis.
First, building on Selberg's approach, 
S. Miquel solved Conjecture \ref{conmar} when $G$
is a product of real rank one groups
in \cite{Miquel17}. 
Second, in the unpublished part of his thesis, 
using a case by case analysis
relying on Hee Oh's approach, 
he was able to deal with 
all simple algebraic group $G$  except
for one family: when  
$G={\rm SO}(2,4n+2)$, $n\geq 1$ and $P$ is the  
stabilizer of an isotropic $2$-plane.
Neither can these methods  deal with  groups $G$ 
with both
real rank one  and higher rank factors.

Pursuing our efforts on that family after the PhD defense,
we found a new strategy which 
does not rely on tricky computations, 
%supersedes the previous approaches, 
avoids the use of Ratner's classification theorem and Weil's rigidity theorem
and gives a full and unified answer to Margulis Conjecture \ref{conmar}.

%13 
\subsection{Strategy of proof}
\label{secstrpro}

\bq 
We now explain the strategy of the proof of Conjecture \ref{conmar} 
and the organization of this paper.
\eq 

\subsubsection{Reflexive commutative and Heisenberg horospherical $U$}

A horospherical subgroup $U$ of $G$ is said to be 
{\it reflexive} if $U$ is conjugate to an
opposite horospherical group $U^-$ (see \S\ref{secdefpar}). 
We will say that $U$ is {\it Heisenberg} if it is two-step nilpotent and if the 
group $P$ acts on the center $\g z$ of the Lie algebra $\g u$ of $U$
by similarities for a suitable Euclidean norm.
This implies that $U$ is reflexive.
More details will be given in Chapter
\ref{secnotati}.

According to reduction steps due to Hee Oh in \cite{Oh98a} and \cite{Oh99}, which we will recall in Chapter \ref{secmaithe}, 
it is enough to prove 
Conjecture \ref{conmar} in two cases:  when $U$ is reflexive commutative and when $U$ is Heisenberg.
One can also assume that $\Gamma$ is a discrete subgroup of $G$
containing both $\exp(\La)$ and $\exp(\La^-)$ 
where $\La$ and $\La^-$ are Lie lattices of
$\g u$ and $\g u^-$. 

\subsubsection{Orbits of horospherical lattices}
The reductive group $L:=P\cap P^{-}$,
intersection of the normalizers of $\g u$ and 
$\g u^-$, acts by conjugation both on the space of lattices 
$\mc X(\g u)$ of $\g u$ and $\mc X(\g u^-)$ of $\g u^-$.
In both cases, 
we will check in Chapter \ref{secorbhor} (Proposition \ref{prosindou}) that 
\begin{equation}
\label{quosindou}
\mbox{\begin{minipage}{24em}
the {\it single $L$-orbit} $L\La$ is closed in 
the space $\mc X(\g u)$
and
the {\it double $L$-orbit} $L\, (\La,\La^-)$ is closed in  
$\mc X(\g u)\times\mc X(\g u^-)$.
\end{minipage}}
\end{equation}
We will be able to prove \eqref{quosindou} without using
Ratner's classification theorem
or Weil's rigidity theorem.
This explains why we can work directly with the group $L$
instead of working as in \cite{Oh98a} with the subgroup $S$ of $L$
generated by its unipotent elements.
Indeed the remaining open cases in \cite{Oh98a}
were exactly those for which 
$S$ is much smaller than $L$. For instance 
in the case where 
$G={\rm SO}(2,n+2)$, $n\geq 4$ and $P$ is the  
stabilizer of an isotropic $2$-plane,
the group $S\simeq {\rm SL}(2,\m R)$ is much smaller
than the group $L\simeq {\rm GL}(2,\m R)\times SO(n,\m R)$. 

The strategy to prove \eqref{quosindou} 
is to introduce a polynomial function $\Ph$ on $G$
which is semi-invariant by multiplication by $L$
and 
such that the sets $\Ph(\Gamma)$ and $\Ph(\Gamma w_0)$
are closed discrete subsets of $\m R$,
where $w_0$ is the longest element of the Weyl group
(Corollary \ref{corphg}).  
We then focus on two polynomial functions given by
$F(X):=\Ph(e^Xw_0)$ for $X$ in $\g u$ and 
$G(X,Y):=\Phi(e^Xe^Y)$ for $X$ in $\g u$ and $Y$ in $\g u^-$.
The former will be used to prove that the single $L$-orbit is closed and the latter
to prove that the double $L$-orbit
is closed. 
The two main properties of these polynomials are their $L$-invariance and the discreteness of
the sets $F(\La)$ and $G(\La,\La^-)$.

\subsubsection{Construction of the function $\Phi$}
Assume first that $\g u$ is reflexive commutative.
The key idea in order to construct $\Ph$ 
is to think of $G$ as a group of $3\times 3$ block-matrices
acting on the natural $3$-terms graduation
$\g g=\g g_{-1}\oplus \g g_0\oplus \g g_1$,
where $\g u^-=\g g_{-1}$, $\g l=\g g_0$ and  $\g u=\g g_{1}$
are the Lie algebras of $U^-$, $L$ and $U$.
We then extract the right-lower block $M(g)\in {\rm End}(\g u)$.
The function $\Ph(g)$ is nothing but the determinant 
$\Phi(g)=\det(M(g))$. 
The fact that $\Phi(\Gamma)$ is a discrete set
follows from an application of the Bruhat decomposition of $G$ and a repeated application of the Mahler compactness criterion.

When $\g u$ is Heisenberg, one uses the same construction
relying on the natural $5$-terms graduation $\g g=\g g_{-2}\oplus \cdots\oplus \g g_2$, 
thinking of $G$ as a group of $5\times 5$ block-matrices
and extracting the right-lower block $M(g)\in {\rm End}(\g z)$.

A key algebraic property of the function $F$
will be checked in Chapter \ref{secapphor}.
It says that the group $L$, seen as a
subgroup of the group ${\rm Aut}_{gr}(\g u)$
of graded automorphisms of $\g u$, 
is, up to finite index, 
equal to the subgroup $H$
that preserve $F$ up to a scalar multiple 
(Proposition \ref{profphf}).
This property is not surprising since, when $\g g$ is absolutely simple, the group 
$L$ is either equal to ${\rm Aut}_{gr}(\g u)$ or is a maximal subgroup of 
${\rm Aut}_{gr}(\g u)$  up to finite index.

\subsubsection{Stabilizer of horospherical lattices and construction of $G_\m Q$}
In Chapter \ref{secarigam}, 
we deduce  from \eqref{quosindou} 
that  
\begin{equation}
\label{eqnlllinf}
\mbox{\rm the stabilizer $L_{\La,\La^-}$ of $(\La,\La^-)$ is infinite,}
\end{equation}
except in one case: when the restricted root system of $G$ is of type $A_2$. This exceptional case is dealt with in Lemma \ref{lemheidic}.$ii$. 
When the parabolic subgroup $P$ is not minimal, the proof of \eqref{eqnlllinf} 
relies on Dani-Margulis' recurrence theorem for unipotent flow (see Proposition \ref{prodanmar}).
When $P$ is minimal, the proof of \eqref{eqnlllinf}  
relies once more on the Mahler compactness criterion
(see Lemmas \ref{lemsod} and \ref{lemheidic}).  
Note that when $G$ is of rank one this stabilizer $L_{\La,\La^-}$ is always finite.
The higher rank assumption is used here through the non-compactness
of the subgroup 
$L_0:=\{\ell\in L\mid \det_\g u({\rm Ad}\ell)=1\}$.

Once \eqref{eqnlllinf} is proven, 
without loss of generality, we can assume that $\Gamma$
contains this infinite group $L_{\La,\La^-}$. 
We denote by $Q$ the Zariski closure of $\Gamma\cap P$.
Its Lie algebra $\g q$ is endowed with a natural $\m Q$-form.
We check that 
the various $\m Q$-forms on the spaces ${\rm Ad}g(\g q)$, for $g$ in $\Ga$,
are induced by a $\Gamma$-invariant $\m Q$-form 
on $\g g$. 
This tells us that $\Gamma$ is included in a $\m Q$-form $G_\m Q$ of $G$ (Proposition \ref{proextfor}). 
Using a result (Proposition \ref{proragven}) of Raghunathan and Venkataramana, we then conclude that $\Ga$ is commensurable to $G_\m Z$ (Corollary \ref{corextfor}).
\vs

The above strategy 
can also be extended to products of semisimple groups over various local fields (see Section \ref{secsarset}).

%2
\section{Notation}
\label{secnotati}

We collect in this chapter 
a few definitions and notations that we will use freely 
all over this paper.

%21 
\subsection{Algebraic groups}
\label{secalggro}

Let ${\bf G}$ be a complex linear algebraic group.
This means that  
${\bf G}$ is a subgroup of 
${\rm GL}(V)$, for some finite dimensional complex vector space $V$ and 
${\bf G}$  is  the set of zeroes 
of a family  $\mc P$ of polynomials on $V$:
$$
{\bf G}=\{g\in {\rm GL}(V)\mid P(g)=0\;\;
\mbox{for all $P$ in $\mc P$}\}
$$
Given a basis $\mc B$ of $V$, we will say that $\bf G$ is defined
over $\m R$ (resp. $\m Q$) if the family $\mc P$ can be chosen
with polynomials with real (resp. rational) coefficients in $\mc B$.
In that case, we will denote by ${\bf G}_{\m R}$ (resp. ${\bf G}_{\m Q}$) the
set of elements of $\bf G$ with real (resp. rational) coefficients.
We will say that the groups ${\bf G}_{\m R}$ and ${\bf G}_{\m Q}$ are
a real form and a $\m Q$-form of $\bf G$, respectively.
We also call the group $G={\bf G}_{\m R}$ a real algebraic Lie group and
denote its Lie algebra by the corresponding gothic letter $\g g$.

This naive approach of algebraic groups, 
avoiding functors and schemes, will
be enough for our purpose.

Let $G$ be a {\it semisimple} real algebraic Lie group. 
This means that the Lie algebra $\g g$ is semisimple i.e.
$\g g$  is a direct sum of 
simple ideals $\g g_a$ where $a$ runs over a finite set 
$\mc A$. 
Since the morphism $G\ra {\rm Aut}(\g g)$ 
has finite kernel and finite index image, 
we will often assume, without loss of generality,
that $G$ is {\it adjoint}. This means that $G$ is 
isomorphic to the Zariski connected component of ${\rm Aut}(\g g)$.
 
The group $G$ or the Lie algebra $\g g$ is said to be {\it absolutely simple} if the complexified Lie algebra
$\g g_\m C$ is simple.
When $\g g$ is simple but not absolutely simple 
it has a complex  structure.

\begin{Exa}
- When ${\bf G}$ is the orthogonal group 
${\bf G}= O(d,\m C)$, 
the groups $G={\rm O}(p,q)$, with $p+q=d$ are 
real forms of ${\bf G}$.\\
- For any quadratic form $Q$ 
with rational coefficients and signature $(p,q)$, the
orthogonal group $O(Q,\m Q)$ is a $\m Q$-form of $G$.
Two quadratic forms $Q_1$ and $Q_2$ give $\m Q$-isomorphic 
$\m Q$-forms if and only if they are proportional.\\
- When $d=2m$, 
the quaternionic orthogonal groups $G=O(m,\m H)$ 
is also a real form of ${\bf G}$. Using quaternion division
algebras over $\m Q$, 
one also constructs many $\m Q$-forms of this  group $G$.\\
- More generally, by a theorem of Borel, 
every non-compact simple real algebraic Lie group $G$ admits at least one 
$\m Q$-form $G_\m Q$ containing non-trivial unipotent elements.\\
- The semisimple real algebraic Lie group $G={\rm SL}(2,\m C)$ is simple but not absolutely simple.
Its complexified Lie algebra is  
$\g g_\m C\simeq \g s\g l(2,\m C)\oplus \g s\g l(2,\m C)$. 
\end{Exa}

%22
\subsection{Parabolic and horospherical subgroups}
\label{secparhor}
\bq
We recall a few properties of parabolic subgroups. 
\eq

\subsubsection{Definition}
\label{secdefpar}
Let $G$ be a Zariski connected semisimple real algebraic Lie group.
A {\it unipotent} subgroup $U$ of $G$ is a subgroup whose elements are unipotent.
Note that the {\it maximal unipotent subgroups} are all conjugate.  
By definition, a {\it minimal parabolic subgroup} $P_0$ of $G$ is 
the normalizer of a maximal unipotent subgroup $U_0$ of $G$.
A {\it parabolic subgroup} $P$ of $G$ is a group that contains a minimal parabolic subgroup $P_0$.
A {\it horospherical subgroup} $U$ of $G$ 
is the unipotent radical of a non-trivial parabolic subgroup $P$
i.e. the largest normal unipotent subgroup of $P$.
The parabolic subgroup $P$ is both the normalizer of $U$ and of its Lie algebra $\g p$.
This Lie algebra $\g p$ of $P$  is called a  {\it parabolic subalgebra} and 
the Lie algebra $\g u$ of $U$ is called a   {\it horospherical subalgebra}.

A parabolic subgroup $P^{-}$ is said to be {\it opposite} to the parabolic subgroup $P$
if  the intersection $L:=P^{-}\cap P$ is  a Levi subgroup 
for both $P$ and $P^-$.
One then has the equality
$P=LU$ and $P^{-}=LU^{-}$.
We will also say that $U^{-}$ is opposite to $U$.
All parabolic subgroups $P^{-}$  opposite to $P$ are conjugate by an element of $U$.
The set 
$
\Omega:=U^{-}LU=\{ g\in G\mid P^-
\;\mbox{\rm is opposite to}\; gPg^{-1}\}
$ 
is then a Zariski open set in $G$. See  \cite{BorelTits65}
for more details.

The group  $U$ or the group $P$  is said 
to be {\it reflexive} if there exists 
an element $g$ in $G$
such that $gUg^{-1}$ is opposite to $U$.
The set of such elements $g$ is also Zariski open in $G$.

\subsubsection{Root systems}
\label{secroosys}
We recall the construction of the {\it standard} minimal parabolic subalgebra $\g p_0$.
Let $A$ be a maximal split torus of $G$ and $\g a$
its Lie algebra. 
The real rank $r$ of $G$ is the dimension of $\g a$.
Let $\Sigma$ be the set of 
{\it restricted roots} i.e. of roots of $\g a$ in $\g g$,
let $\Sigma^+\subset \Sigma$ be a choice of positive roots
and $\Pi\subset\Sigma^+$ be the corresponding set of simple roots. 
This set $\Pi=\{\al_1,\ldots ,\al_r\}$ is a basis of $\g a$. Let $W$ be the Weyl group of $\Sigma$ 
and $w_0$ be the element of $W$ such that $w_0(\Pi)=-\Pi$. For every $\al\in \Sigma\cup\{0\}$, let 
$\g g_\al$ be the corresponding weight space so that one has
$$
\g g= 
\textstyle
\bigoplus\limits_{\al\in \Sigma\cup\{0\}}
\g g_\alpha,
\;\;\mbox{\rm
and}\;\;
\g u_{_\Pi}:= \textstyle\bigoplus\limits_{\al\in \Sigma^+}\g g_\alpha
\;\;{\rm and}\;\;
\g p_{_\Pi}:= \textstyle\bigoplus\limits_{\al\in \Sigma^+\cup\{0\}}\g g_\alpha
$$
are respectively a maximal unipotent subalgebra and a minimal parabolic subalgebra of $\g g$.

We now explain the construction of the finitely many 
parabolic subalgebras $\g p$ containing $\g p_{_\Pi}$. 
They are parametrized by the subsets $\theta$ of $\Pi$.
For such a subset $\th\subset\Pi$, we introduce the function 
$n_\th:\Sigma\cup\{0\}\ra \m Z$ defined by 
\begin{equation}
\label{eqnnth}
\textstyle
n_\th(\,\sum\limits_{  i\leq r}k_i\al_i\,):=\sum\limits_{\al_i\in \th}k_i
\; ,\;\;\;\mbox{\rm for all $k_i$ in $\m Z$}.
\end{equation}
The Lie subalgebras 
\begin{equation}
\label{eqnuthpth}
\g u_\th:= \textstyle\bigoplus\limits_{n_\th(\al)>0}\g g_\alpha
\;\;{\rm and}\;\;
\g p_\th:= \textstyle\bigoplus\limits_{n_\th(\al)\geq 0}\g g_\alpha
\end{equation}
are the standard horospherical Lie subalgebra 
and the standard parabolic Lie subalgebra associated to 
$\th$.
The Lie subalgebras
\begin{equation}
\g u^{-}_\th:= \textstyle\bigoplus\limits_{n_\th(\al)<0}\g g_\alpha
\;\;{\rm and}\;\;
\g p^{-}_\th:= \textstyle\bigoplus\limits_{n_\th(\al)\leq 0}\g g_\alpha
\end{equation}
are the standard opposite horospherical Lie subalgebra 
and the standard opposite parabolic Lie subalgebra.
The intersection $\g l_\th:=\g p_\th\cap \g p^{-}_\th$ is a Levi 
subalgebra of both $\g p_\th$ and $\g p^{-}_\th$.
The horospherical Lie subalgebra $\g u_\th$ 
is reflexive if
and only if one has $w_0(\th)=-\th$.

Note that any pair of opposite horospherical 
subalgebras $(\g u, \g u^{-})$ of $\g g$ is equal
to a pair $(\g u_\th, \g u^{-}_\th)$ given by a 
suitable choice of $A$, $\Sigma^+$ and $\th$.
The classification of reflexive commutative 
horospherical subalgebras will be recalled in 
Section \ref{secrefcom}.

\begin{Def}
\label{defroospa} 
One says that {\it the center $\g z$ of $\g u$ is a root space}
if the action of $A$ on $\g z$ is scalar.
\end{Def}

This means that $\g u$ is included in a simple ideal of $\g g$ 
and $\g z$ is the root space
$\g g_{\widetilde{\al}}$ associated to the corresponding 
largest root 
$\widetilde{\al}$ of 
$\Sigma$.

\begin{Def}
\label{defheihor} 
We will say that the horospherical group $U$ is Heisenberg
if $\g g$ is simple,
$\g u$ is two-step nilpotent and, 
the center $\g z$ of $\g u$ is a root space.
\end{Def}

The Heisenberg horospherical Lie subalgebras are always reflexive.
Their classification will be recalled in 
Section \ref{secheihor}. 

\subsubsection{Graded Lie algebras}
\label{secgralie}
The structure of horospherical Lie subalgebras can be understood using the
graduation of $\g g$ induced by the function $n_\theta$.
The following lemma on root systems is classical. 
\begin{Lem}
\label{lemgralie}
$a)$ If $\al$, $\be$ and $\al+\be$ are roots, then 
$[\g g_\al,\g g_\be]=\g g_{\al+\be}$.\\
$b)$ For any positive roots $\be_1,\ldots ,\be_k$ whose 
sum $\be_1+\cdots +\be_k$ is a  root, 
there exists a permutation $\si$ of $[1,k]$
such that  all  
$\be_{\si(1)}+\cdots +\be_{\si(j)}$ are roots.\\
$c)$ If $\g g$ is simple and $\be_1$ is a positive root,
there exists a sequence $\be_2,\ldots ,\be_k$ of simple roots
such that all  
$\be_{1}+\cdots +\be_{j}$ are roots and 
$\be_{1}+\cdots +\be_{k}$ is the largest root $\widetilde\al$.
\end{Lem}
\begin{proof} 
$a)$ One can assume $\g g$ complex and apply \cite[Prop.4 \S VIII.2.2]{Bourbaki78}. 

$b)$ This is \cite[Prop. 19 \S VI.1.6]{Bourbaki456}.

$c)$ The adjoint representation 
has a unique highest weight $\widetilde\al$.
\end{proof} 
Let $s:=\max\{n_\th(\al)\mid \al\in \Sigma\}$.
For $j\in \m Z$, we denote by 
$
\g g_j:=\textstyle\bigoplus\limits_{n_\th(\al)=j}\g g_\alpha 
$ 
so that one has 
\begin{equation}
\label{eqngsgogs}
\g g=\g g_{-s}\oplus\ldots \oplus \g g_0
\oplus\ldots \oplus \g g_s\, ,
\end{equation}
\begin{equation*}
\g l_\th =\g g_0\; ,\; \g u_\th=\g g_1
\oplus\ldots \oplus \g g_s\; ,\;\;
\g p_\th= \g l_\th\oplus \g u_\th\, ,
\end{equation*}
\begin{equation*}
\g u^{-}_\th=\g g_{-s}
\oplus\ldots \oplus \g g_{-1}\;\; ,\;\;
\g p^{-}_\th= \g l_\th\oplus \g u^{-}_\th\, .
\end{equation*}

For latter use, the following lemma will be useful,

\begin{Lem}
\label{lemnonroo}
$a)$ One has the equality $[\g g_1, \g g_j]=\g g_{j+1}$
for all $j\geq 1$.\\
$b)$ The Lie algebra $\g u_\th$ is $s$-step nilpotent.\\
$c)$ The  center of $\g u_\th$ is 
$\g z=\g g_s$.\\ 
$d)$ When $\g g$ is simple, the Lie subalgebra 
$\g g'=\g g_{-s}\oplus [\g g_{-s},\g g_s]\oplus \g g_s$
is also simple.
\end{Lem}

\begin{proof} 
We can assume $\g g$ complex and simple.

$a)$ and $b)$ follow from Lemma \ref{lemgralie}.a and  \ref{lemgralie}.b

$c)$ and $d)$ follow from Lemma \ref{lemgralie}.a and \ref{lemgralie}.c and from \cite[Prop.2 \S VIII.3.1]{Bourbaki78} for the semisimplicity of $\g g'$.
\end{proof}

\begin{Rem}
For $\theta\subset\Pi$ we introduce the 
element $h_\th\in \g a$ by 
\begin{equation}
\label{eqnhth}
\al_i(h_\theta)=1\;\;{\rm if}\; \al_i\in \th
\;\;\;\;\; {\rm and}\;\;\;\;\;
\al_i(h_\theta)=0\;\;{\rm if}\; \al_i\not\in \th,
\end{equation}
so that one has the equality $n_\th(\al)=\al(h_\th)$
for all $\al\in \Sigma\cup\{0\}$.
In the decomposition \eqref{eqngsgogs} 
the spaces $\g g_j$ are
the eigenspaces of 
${\rm ad}h_\th$ for the eigenvalue $j$.
\end{Rem}

%23
\subsection{Examples of horospherical subgroups}
\label{secexahor}
\bq
To be concrete we now discuss explicit examples. 
\eq

Let $G={\rm SL}(d,\m R)$, $d\geq 3$
and
$\g a=\{ x={\rm diag}(x_1,\ldots, x_d)\;/\;
{\rm tr}(x)=0\}$. The set of roots is
$\Si=\{e^*_i-e^*_j ,\; 1\leq i\neq j\leq d\}$.
and the root spaces $\g g_{e^*_i-e^*_j}$
are  one dimensional with basis $E_{i,j}=e_j^\star\otimes e_i$.
We take
$\Si^+=\{e^*_i-e^*_j,\; 1\leq i<j\leq d\}$ so that
$\Pi=\{\al_i:=e^*_{i+1}-e^*_i ,\; 1\leq i< d\}$. 
The corresponding maximal horospherical and minimal parabolic subalgebras are,
say when $d=4$,
$$
\g u_{_\Pi}=
\mbox{\scriptsize$
\left\{ \left(
\begin{array}{cccc}
0&*&*&*\\
0&0&*&*\\
0&0&0&*\\
0&0&0&0
\end{array}
\right) \right\}
$}
, \hspace{1em}
\g p_{_\Pi}=
\mbox{\scriptsize $
\left\{ \left(
\begin{array}{cccc}
*&*&*&*\\
0&*&*&*\\
0&0&*&*\\
0&0&0&*
\end{array}
\right) \right\}
$}
.
$$
We now describe the horospherical Lie algebras $\g u_\th$ in terms of block matrices.

When the set $\th$ contains only one simple root,
$\th=\{\al_{d_1}\}$, one has  a block decomposition given by 
the equality $d=d_1\!+\! d_2$:
$$
\g u_\th=
\mbox{\footnotesize $
\left\{ \left(
\begin{array}{c|c} 0 &B   \\
\hline 
  0 &0  
\end{array}
\right) \right\}
$}
, \hspace{1em}
\g l_\th=
\mbox{\footnotesize $
\left\{ \left(
\begin{array}{c|c} A & 0 \\ 
\hline
0 & D  
\end{array}
\right) \right\}
$}
, \hspace{1em}
\g u^-_\th=
\mbox{\footnotesize $
\left\{ \left(
\begin{array}{c|c} 0 & 0  \\ 
\hline
{}C & 0  
\end{array}
\right) \right\}
$}
.
$$
This horospherical Lie algebra $\g u_\th$ is reflexive when $d_1=d_2$.
Conjecture \ref{conmar} for this case is already proven in   \cite{Oh98a}, but our strategy for this reflexive commutative case is new even when $d_1=d_2=2$.
The reader is advised to keep this particular example in mind.
\vs

When the set $\th$ contains two simple roots,
$\th=\{\al_{d_1},\al_{d_1+d_2}\}$,  
one has  a block decomposition given by 
the equality $d=d_1\!+\! d_2\!+\! d_3$:
$$
\g u_\th=
\mbox{\footnotesize $
\left\{ \left(
\begin{array}{c|c|c} 0 & * & * \\
\hline 
0 & 0 & * \\ 
\hline
0 & 0 & 0 \end{array}
\right) \right\}
$}
, \hspace{1em}
\g l_\th=
\mbox{\footnotesize $
\left\{ \left(
\begin{array}{c|c|c} * &0 & 0 \\ 
\hline
0 & * & 0 \\ 
\hline
0 & 0 & * \end{array}
\right) \right\}
$}
, \hspace{1em}
\g u^-_\th=
\mbox{\footnotesize $
\left\{ \left(
\begin{array}{c|c|c} 0 & 0 & 0 \\ 
\hline
{}* & 0 & 0 \\ 
\hline
{}* & * & 0 \end{array}
\right) \right\}
$}
.
$$

This horospherical algebra $\g u_\th$ is reflexive when $d_1=d_3$,
and is Heisenberg when $d_1=d_3=1$.
Conjecture \ref{conmar} for this case is already proven in \cite{BenoistOh10b} when $d_2=1$
and in \cite{Oh98a} when $d_2\geq 2$, but
our strategy for this Heisenberg case 
is new even for $d_2=1$ or  $2$.
\vs

Given a horospherical subgroup $U$ of $G$, 
there may exist  many arithmetic lattices $\Ga$ that intersect $U$ cocompactly. Here are a few examples:

\begin{Exa}
When $G={\rm SL}(4,\m R)$ and 
$U=  \{ \mbox{\footnotesize $
\left(
\begin{array}{c c} I &X   \\
  0 &I  
\end{array}
\right)$}\mid X\in \mc M(2,\m R)  \}
$, the following arithmetic lattices $\Ga$   intersect $U$ cocompactly:\\
$(i)$
$\Ga= {\rm SL}(4,\m Z)$,\\ 
$(ii)$ $\Ga={\rm SU}(h,\m Z[\sqrt{2}])$
where $h$ is a non-degenerate hermitian form on $\m Q[\sqrt{2}]^4$.\\
$(iii)$ $\Ga ={\rm SL}(2,\m H_\m Z)$ where $\m H_\m Q$ 
is a quaternion division algebra over $\m Q$.
\end{Exa}

\begin{Exa}
When $G={\rm SO}(1,4)\times {\rm SO}(5,\m C)$ and 
$P$ is the product of  the stabilizers of isotropic lines,
the group $U$ is commutative reflexive and isomorphic to $\m R^3\times \m C^3$.
The following group can be embedded in $G$ as an arithmetic lattice:
$\Ga= {\rm SO}(q,\m Z[\sqrt[3]{2}])$
where $q=x_1^2+ x_2^2
+x_3^2+x_4^2- x_5^2$.
\end{Exa}

%24
\subsection{Discrete subgroups}
\label{secdissub}

\bq
We now collect results on discrete subgroups of Lie groups.
\eq

\subsubsection{Arithmetic lattices}
Let $G$ be a real algebraic Lie group.
Two discrete subgroups $\Gamma_1$ and $\Gamma_2$
are said to be {\it commensurable} if the group
$\Gamma_1\cap \Gamma_2$ has finite index in 
both $\Gamma_1$ and $\Gamma_2$.
A discrete subgroup $\Ga$ of $G$ is a {\it lattice}
if the $G$-invariant 
measures on the quotient $G/\Ga$ have finite volume.
When $G$ is semisimple, a 
discrete subgroup $\Ga\subset G$ is said to be {\it irreducible} 
if, for all proper normal subgroups $G'$ of $G$,
the intersection $\Ga\cap G'$ is finite. 

Assume now that $G$ is semisimple adjoint. 
A non-cocompact irreducible discrete subgroup 
$\Gamma$ of 
$G$ is called an {\it arithmetic subroup} of $G$
if there exists a $\m Q$-form $G_\m Q$ of $G$ such that $\Gamma$ is commensurable
with $G_{\m Z}$.
Such a subgroup 
is always a lattice.
See \cite{Borel69} for more detail.

\begin{Rem}
The definition of an arithmetic subgroup $\Gamma$ 
of a semisimple adjoint real algebraic group $G$ used in Margulis arithmeticity theorem 
is more general. 
It uses an auxiliary compact semisimple real algebraic Lie group $K$ 
and a $\m Q$-form $H_\m Q$ of the product $H:=G\times K$.
A group $\Gamma$ is defined to be
arithmetic if it is commensurable with
the projection in $G$ of $H_\m Z$.
When $\Gamma$ is irreducible and the group $K$ is non-trivial all the groups constructed in this way are 
cocompact. This explains why no auxiliary group $K$
is involved in our definition. 
\end{Rem}

\begin{Rem}
When the semisimple group $G$ is not adjoint, 
we define an arithmetic subgroup of $G$
as a subgroup whose image in the adjoint group is 
arithmetic. Note that the $\m Q$-structure on the adjoint group might not lift to $G$. 
\end{Rem}

\subsubsection{Nilpotent lattices}

\begin{Lem}
\label{lemnillat}
Let $U$ be a unipotent real algebraic Lie group, 
$\g u$ be its Lie algebra, 
and $\Delta$ a lattice of $U$.\\
$a)$ There exists a unique $\m Q$-form $U_\m Q$ of $U$
such that $\Delta\subset U_\m Q$.\\
$b)$ The set $\g u_\m Q:=\log U_\m Q$ is then a $\m Q$-form of the Lie algebra $\g u$.\\
$c)$ Two lattices $\Delta$ and $\Delta'$ give the same $\m Q$-form
iff they are commensurable.
\end{Lem}
A lattice $\Lambda$ in a Lie algebra $\g u$ is 
a {\it Lie lattice} if $[\Lambda,\Lambda]\subset \Lambda$.

We denote by $C^k\Delta$ and $C^kU$ the descending central sequences of $\Delta$ and $U$.
\begin{Cor}
\label{cornillat} 
$a)$ There exists a Lie lattice $\Lambda$ of $\g u$
such that ${\rm exp}(\Lambda)\subset \Delta$.\\
$b)$ For all $k\geq 1$, the group $C^k\Delta$ is a lattice of $C^kU$.\\
$c)$ Let $U'\subset U$ be a subgroup such that $\Delta\cap U'$ 
is a lattice in $U'$ and $U''$ be the centralizer of $U'$ in $U$,
then $\Delta\cap U''$ is a lattice in $U''$.  
\end{Cor}

\begin{proof}
See \cite[Chapter 2]{Raghunathan72}.
\end{proof}

We will also need the following lemma.

\begin{Lem}
\label{lemactlat}
Let $\g u$ be a  nilpotent real Lie algebra,
$H$ be the group of automorphisms of $\g u$
and $\Lambda$ be a Lie lattice of $\g u$. Then the orbit 
$H\Lambda$ is closed in the space $\mc X(\g u)$ of lattices of $\g u$.
\end{Lem}

\begin{proof} 
Let $\ph_n$ be a sequence of automorphisms of $\g u$ such that 
the sequence of Lie lattices $\La_n:=\ph_n(\La)$ converges to a lattice $\La_\infty$. Since the  lattices
$\La_n$ are Lie lattices, the lattice $\La_\infty$ is also a Lie lattice. We want to find an automorphism $\psi $ of $\g u$
such that $\La_\infty=\psi(\La)$.

Let $(X_1,\ldots,X_d)$ be a basis of $\Lambda_\infty$.
For all $j\leq d$, the exists  $X_{j,n}\in \La_n$
such that the sequence $X_{j,n}$ converges to $X_j$.
For $n$ large the family $(X_{1,n},\ldots,X_{d,n})$ is a basis of $\Lambda_n$.
For all $i, j\leq d$, one has a relation 
$$
\textstyle
[X_i,X_j]=\sum_{k\leq d} c^k_{i,j}X_k
\;\;{\rm with}\;\; c^k_{i,j}\in \m Z.
$$ 
Note that there exists a neighborhood $\Om_0$ of $0$ in $\g u$ 
such that, for all $n\geq 1$, one has $\La_n\cap \Om_0=\{0\}$.
Therefore, for $n$ large, the same relations 
$$
\textstyle
[X_{i,n},X_{j,n}]=\sum_{k\leq d} c^k_{i,j}X_{k,n}
$$
must be satisfied. The linear map 
$\psi_n: \g u\ra \g u$ given by $\psi_n(X_{j,n})=X_j$ for all
$j\leq d$
is then an automorphism of $\g u$ such that $\psi_n(\La_n)=\La_\infty$.
\end{proof}

\subsubsection{Horospherical lattices}
Let  $U$ be a horospherical group in a Zariski connected
semisimple real algebraic Lie group $G$.

\begin{Def}
\label{defirrlat}
- A lattice $\Delta$ of $U$ is {\it irreducible} if for any proper normal subgroup $G'$
of $G$, 
one has $\Delta\cap G'=\{e\}$.\\
- A lattice $\Delta$ of $U$ is {\it indecomposable} 
if one cannot write $G$ as a product $G=G'G''$ 
of two proper normal subgroups 
with finite intersection such that  the group $(\Delta\cap G')(\Delta\cap G'')$  has finite index in $\Delta$.\\
These notions depend on the embedding $U\hookrightarrow G$.
\end{Def}

For lattices $\Delta\subset U$ included in discrete Zariski dense subset $\Gamma$ of $G$,
we will see in Lemma \ref{lemirrhor} that these two notions are equivalent.

The following lemma tells us that when working with a Zariski dense subgroup $\Gamma$
containing a lattice in a reflexive horospherical group,
one can assume that $\Gamma$ contains another lattice in an opposite
horospherical subgroup.

\begin{Lem}
\label{lemgamuum}
Let $G$ be a Zariski connected semisimple real algebraic Lie group,
$U$ a non-trivial reflexive horospherical subgroup, 
and $\Gamma$ a  discrete subgroup 
of $G$ that contains an irreducible lattice $\Delta$ of $U$.
One has the equivalence \\
$i)$ $\Gamma$ is Zariski dense,\\ 
$ii)$ $G$ has no compact factors and there exists a horospherical 
subgroup $U^-$ opposite to $U$ such that $\Gamma$ also 
contains an irreducible lattice $\Delta^-$ of $U^-$.
\end{Lem}

\begin{proof}
$i)\Longrightarrow ii)$ Since $\Delta$ is irreducible, the group $G$ has no compact factors.

Since $U$ is reflexive, the set $\{g\in G\mid gUg^{-1}\;
\mbox{\rm is opposite to $U$}\}$ is non-empty and Zariski open
in $G$. Therefore it contains an element $g_0$ of $\Gamma$
and the group $\Gamma$ contains the lattice $g_0\Delta g_0^{-1}$
of the opposite group $g_0Ug_0^{-1}$

$ii)\Longrightarrow i)$ The lattices $\Delta$ and $\Delta^-$ 
are Zariski dense respectively in $U$ and $U^-$.
Since $\Delta$ is irreducible, the groups $U$ and $U^-$
intersect all the simple factors of $G$, hence they generate
the group $G$.
\end{proof}

We can now state the main theorem of this paper:

\begin{Thm}
\label{thmmar}
Let $G$ be a semisimple real algebraic Lie group 
of real rank at least $2$ and $U$ 
be a non-trivial horospherical subgroup of $G$.
Let $\Gamma$ be a discrete Zariski dense subgroup of $G$
that contains an indecomposable lattice $\Delta$  of $U$. Then $\Gamma$
is a non-cocompact irreducible arithmetic lattice of $G$.
\end{Thm}

%3
\section{Orbits of horospherical lattices}
\label{secorbhor}

In this chapter we focus on lattices in reflexive horospherical subgroups which are 
either commutative or Heisenberg. Our aim is to prove that 
the single $L$-orbit  and   the double $L$-orbit 
in the space of horospherical lattices are closed
(Proposition \ref{prosindou}). More precisely we will explain in this chapter the dynamical part
of the proof, and postpone to Chapter \ref{secapphor} the algebraic part (Proposition \ref{profphf}.a and Lemma \ref{lemfphf}).

%31 
%\subsection{Horospherical lattices}
%\label{secacthor}
We will keep the following notation throughout this chapter.

\begin{Not}
\label{notacthor}
Let $G$ be a semisimple adjoint real algebraic Lie group. 
Let 
$P$ and $P^{-}$ be two opposite parabolic subgroups,
$U$ and $U^{-}$ be their unipotent radicals, and $L:=P\cap P^{-}$ so that 
$P=LU$ and  $P^{-}=LU^{-}$.
Let $\g g$, $\g p$, $\g p^-$, $\g u$, $\g u^-$ and $\g l$
be the corresponding Lie algebras.
Let $\Delta$ be a lattice of $U$ and $\Delta^{-}$ be a lattice of $U^{-}$
such that the subgroup $\Gamma$ of $G$ generated by 
$\Delta$ and $\Delta^-$
is discrete.
Let $\Lambda$ and $\Lambda^{-}$ be  Lie lattices of $\g u$  and    $\g u^{-}$ such that
$\exp(\Lambda)\subset\Delta$ and 
$\exp(\Lambda^-)\subset\Delta^-$.
\end{Not}

In the next two sections, we begin with generic constructions which do not use the assumption that $\g u$ is reflexive commutative or Heisenberg.

%31
\subsection{The matrices $M(g)$}
\label{secmatmg}

\bq
The key idea is to think of the elements $g$ of $G$ as block matrices and to extract
a suitable block $M(g)$.
\eq

Let $\g z$ be the center of $\g u$ and 
$\pi: \g g\ra \g z$  be the $L$-equivariant projection:
in the decomposition \eqref{eqngsgogs}, $\pi$ is the projection on the last factor.

\begin{Def}
\label{defmg}
For $g$ in $G$, we set 
$M(g):=\pi\,{\rm Ad}g\,\pi \in {\rm End}(\g z)$.
\end{Def}
This matrix $M(g)$ is the lower-right block of ${\rm Ad} g$ in the decomposition \eqref{eqngsgogs}. 
Notice that the map $M:G\ra {\rm End}(\g z)$ is a polynomial map.

Here are two important features of the projection $\pi$ and the map $M$.

\begin{Lem}
\label{lemmgxadl}
$a)$ For $v$ in $U^-$ and $u$ in $U$, 
one has 
\begin{equation}
\label{eqnadvupi}
\pi\,{\rm Ad}v={ \rm Ad}u\, \pi=\pi\, .
\end{equation} 

$b)$ When $g$ is in the Zariski open set $\Om:=U^-LU$, i.e. when $g=v\ell u$ with $v\in U^-$, 
$\ell\in L$, $u\in U$, one has
\begin{eqnarray}
\label{eqnmgxadl}
M(g)X&=&{\rm Ad}\ell\, X ,
\;\;\mbox{\rm  for all $X$ in $\g z$\,}.
\end{eqnarray}
\end{Lem}
\begin{proof}
$a)$ is clear and $b)$ follows from $a)$. 
\end{proof}

\begin{Prop}
\label{promgx}
With the notation  \ref{notacthor}. The set 
$$
\{M(g)X\mid g\in \Gamma\cap\Omega\, ,\; X\in \Lambda\cap\g z\}
$$
is a closed discrete subset of $\g z$.
\end{Prop}

\begin{proof}
Assume by contradiction that there exists a  sequence
of distinct elements 
$$
X'_n=M(g_n)X_n
\;\;{\rm with}\;\;
g_n\in \Gamma\cap \Omega
\;\;{\rm and}\;\;
X_n\in \Lambda\cap \g z\, 
$$ 
converging to an element $X'_\infty\in \g z$. Since $g_n$ is in $\Omega$, one can write 
$$
g_n=v_n\ell_nu_n
\;\;{\rm with}\;\;
v_n\in U^-\, ,\; \ell_n\in L\, ,\; u_n\in U\,.
$$
Since $\Delta^-$ is cocompact in $U^-$, after extraction, 
we can write 
$$
v_n=\delta_n ^{-1}v'_n
\;\;{\rm with}\;\;
\delta_n\in \Delta
\;{\rm and} \;\;
v'_n\in U^- 
\;\mbox{\rm converging to}\; v'_\infty
\, .
$$
Remembering that $\g z$ is the center of $\g u$ and using \eqref{eqnmgxadl}, one computes 
\begin{equation}
\gamma_n:=\delta_n g_n e^{X_n}g_n^{-1}\delta_n^{-1}\\
=v'_n\ell_n e^{X_n}\ell_n^{-1}{v'_n}^{-1}
=
v'_n e^{X'_n}{v'_n}^{-1}
\end{equation}
Therefore this sequence $\gamma_n$ of elements of $\Gamma$ converges to
$
\gamma_\infty:=v'_\infty e^{X'_{\infty}}{v'_\infty}^{-1}
$
Since $\Gamma$ is discrete, one must have
$\ga_n=\ga_\infty$ for $n$ large. Since the exponential map restricted to the nilpotent elements of $\g g$ is injective, one deduces
\begin{equation}
\label{eqnadvnvi}
{\rm Ad}(v'_n)\, X'_n = {\rm Ad}(v'_\infty)\, X'_\infty
\;\; \;
\mbox{\rm for $n$ large.}
\end{equation}
Applying the projection $\pi$ to Equality \eqref{eqnadvnvi}
and using  \eqref{eqnadvupi},  one concludes $X'_n=X'_\infty$ for $n$ large. Contradiction.
\end{proof}

The following corollary is worth mentioning 
even though we will not use it in the rest of the text.

\begin{Cor}
\label{cormgg}
With the notation  \ref{notacthor}. The set 
$
M(\Gamma\cap \Omega)
$ 
is a closed discrete subset of ${\rm End}(\g z)$.
\end{Cor}

\begin{proof}
The intersection $\Lambda \cap \g z$ is a lattice in $\g z$.
\end{proof}

%32
\subsection{The function $\Phi(g)$}
\label{sefunphg}
\bq
We now introduce the polynomial  function $\Ph$ on $G$.
\eq

\begin{Def} Let $\Phi:G\ra \m R$ be the function on $G$ given by
$$\textstyle
\Phi(g):=\det_{\g z}(M(g))
$$  
\end{Def}

\begin{Prop}
\label{prophg}
With the notation  \ref{notacthor}. The  
set 
$\Phi(\Gamma\cap\Omega) $
is a closed discrete subset of $\m R$.
\end{Prop}

\begin{proof} 
Assume by contradiction there exists a  sequence
of distinct real numbers $\Phi(g_n)$ with 
$g_n\in \Gamma\cap\Omega$
converging to $\Phi_\infty\in \m R$.
Since $g_n$ is in $\Gamma\cap \Omega$, by \eqref{eqnmgxadl}, the matrices 
$M(g_n)$ are invertible and,
by Proposition \ref{promgx}, the following union of lattices of $\g z$
$$
\textstyle\bigcup_{n\geq 1}M(g_n)(\Lambda\cap \g z)
$$
is a closed discrete subset of $\g z$.
In particular, one has
\begin{equation}
\label{eqninfmin}
\inf_{n\geq 1}\min_{X\in \Lambda\cap \g z\smallsetminus 0}
\| M(g_n)X\| >0\, .
\end{equation}
Since the determinants $\Phi(g_n)$ are bounded, 
one has a uniform upper bound on the covolume of these lattices:
\begin{equation}
\label{eqnsupcov}
\sup_{n\geq 1} {\rm covol}
(M(g_n)(\Lambda\cap \g z))\; <\;\infty\, .
\end{equation}
Conditions \eqref{eqninfmin} and \eqref{eqnsupcov}
tell us that this family of lattices of $\g z$ satisfy the
Mahler compactness criterion.
Therefore, after extraction, the sequence 
of lattices $M(g_n)(\Lambda\cap \g z)$ converges 
to a lattice $\Lambda_\infty$ of $\g z$.
Let $X_1,\ldots ,X_d$ be a basis of $\Lambda_\infty$.
For all $j\leq d$, there exists $X_{n,j}\in \Lambda\cap\g z$ such that 
$$
M(g_n)(X_{n,j})
\;\; \mbox{\rm converges to } \;
X_j\, .
$$
Therefore, by Proposition \ref{promgx}, one has 
$M(g_n)X_{n,j}=X_j$ for $n$ large. 
Hence, one has the equality of lattices 
$M(g_n)(\Lambda\cap \g z)=\Lambda_\infty\,.$
Computing the covolume of these lattices gives for $n$ large
$$
|\Phi(g_n)|\,{\rm covol}(\Lambda\cap \g z)=
{\rm covol}(\Lambda_\infty).
$$
Hence the sequence $|\Phi(g_n)|$ is constant for $n$ large. 
Therefore, $\Phi(g_n)$ can take only finitely many values.
Contradiction.
\end{proof}

%33
\subsection{The function $\Phi(gw_0)$}
\label{sefunphgw}
\bq
We assume up to the end of this chapter
that $\g u$ is reflexive commutative or Heisenberg.
This allows us to say more on the function $\Phi$.
\eq

We will use the notation of \S\ref{secroosys}
and \ref{secgralie}. We write $\g u=\g u_\th$
for a subset $\th$ of the set of simple roots
and let $w_0$ be the longest element of the Weyl group so that one has
\begin{equation}
\label{eqnwouwou}
{\rm Ad}w_0(\g u)=\g u^-
\;\;{\rm and}\;\;
{\rm Ad}w_0(\g u^-)=\g u\, .
\end{equation}

In this case, the parabolic subgroup $P$ is maximal among the reflexive parabolic subgroups of $G$. 
Indeed, the set $\th$ contains either one 
simple root invariant by $-w_0$
or a pair of roots exchanged by $-w_0$.
\vs

Here are a few useful properties of the polynomial function $\Ph$.
Let $\chi$ be the character of $L$ given by, for $\ell$ in $L$,
\begin{equation}
\label{eqnchidet}
\chi(\ell)={\rm det}_{\rm \g z}({\rm Ad}\ell).
\end{equation}

\begin{Lem}
\label{lemphg}
With the notation \ref{notacthor}.
Assume that $\g u$ is reflexive
commutative or Heisenberg.\\
$a)$ For all $g$ in $G$, $\ell$, $\ell'$ in $L$, $u$ in $U$, $v$ in $U^-$, one has
\begin{equation}
\label{eqnphlvglu}
\Phi( v\ell g\ell' u)=\chi(\ell)\chi(\ell')\, \Phi(g)\, .
\end{equation}
$b)$ 
For $g\in G$, one has the equivalence:\;  $g\in \Omega\Longleftrightarrow\Phi(g)\neq 0 \, .$
\end{Lem}

\begin{proof} $a)$ By \eqref{eqnadvupi}  one has the equality:
$M( v\ell g\ell' u)={\rm Ad}\ell\, M(g)\,{\rm Ad}\ell'$.

$b)$ 
If $g$ is in $\Omega$, one can write $g=v\ell u$
with $v\in U^-$, $\ell\in L$, $u\in U$, and one has
$\Phi(g)=\chi(\ell)\neq 0\, .$

By the Bruhat decomposition, 
one can write 
$g=v w\ell u$ with $v\in U^-$, $w\in W$, $\ell\in L$, $u\in U$. If $g$ is not in $\Omega$, one has 
${\rm Ad}w\,\g u\neq \g u$.
The normalizer $Q$ of $\g z$ in $G$ 
is a subgroup of $G$ containing $P$.
This group $Q$ is a reflexive parabolic subgroup of $G$.
Since the normalizer $P$ of $\g u$ is maximal among the reflexive parabolic subgroups of $G$, one has
$P=Q$. Therefore, one also has 
${\rm Ad}w\,\g z\neq \g z$.
Since $W$ permutes the weight spaces $\g g_\al$,
this implies that the matrix $M(w)=\pi\, {\rm Ad}w\, \pi\in {\rm End}(\g z)$ is not invertible. 
Hence one has $\Phi(w)=0$ and, by $a)$, one also has $\Phi(g)= 0$.
\end{proof}

\begin{Cor}
\label{corphg}
With the notation \ref{notacthor}. 
Assume that $\g u$ is reflexive commutative or Heisenberg.\\
$a)$ The set 
$\Phi(\Gamma) $
is a closed discrete subset of $\m R$.\\
$b)$  The set $\Phi(\Gamma w_0) $
is a closed discrete subset of $\m R$.
\end{Cor}

\begin{proof}
$a)$ This follows from Proposition \ref{prophg} and Lemma \ref{lemphg}

$b)$ Even though the element $w_0$ might not be 
in $\Ga$, we will deduce point $b)$ from point $a)$. Indeed, since $\Ga$
is Zariski dense in $G$, there exists an element $g_0$ in $\Ga$ 
such that ${\rm Ad}g_0(\g u)=\g u^-$.
Since one has the equality 
${\rm Ad}(g_0^{-1}w_0)(\g z)=\g z$
the following endomorphism $M_0$ is invertible
$$
M_0:=\pi\,{\rm Ad}(g_0^{-1}w_0)\, \pi
\in {\rm GL}(\g z)
$$ 
and we compute, for all $g\in G$,
$$
M(gw_0)= 
M(gg_0)M_0,
$$
and, setting $m_0:={\rm det}_\g z(M_0)$,
we deduce 
$$
\Ph(gw_0)=m_0\,\Ph(gg_0)\, .
$$
Therefore the set $\Ph(\Ga w_0)=m_0\,\Ph(\Ga)$ is  a closed discrete
subset of $\m R$.
\end{proof}

%34
\subsection{The polynomials $F(X)$ and $G(X,Y)$}
\label{secpolfx}
\bq
We introduce  in this section the polynomials 
$F(X)$ on $\g u$ and $G(X,Y)$ on $\g u\times \g u^-$
and study the pairs of automorphisms of $\g u$
and $\g u^-$
that preserve the polynomials 
$F(X)$ and $G(X,Y)$.
Note that this purely algebraic section 
does not involve the group $\Gamma$.
\eq

\begin{Def}
\label{defpolfg}
Let $F:\g u\ra \m R$ and $G:\g u\times \g u^-\ra\m R$
be the polynomial functions given by, for $X$ in $\g u$ and $Y$ in $\g u^-$,  
\begin{equation}
\label{eqnpolfg}
F(X):=\Phi(e^Xw_0)
\;\;\;{\rm and}\;\;\;
G(X,Y):=\Phi(e^Xe^Y).
\end{equation}
\end{Def}

\begin{Cor}
\label{corpolfg}
With the notation \ref{notacthor}. 
Assume that $\g u$ is reflexive commutative or Heisenberg.\\
$a)$ The set $F(\La)$ is a closed discrete subset of $\m R$.\\
$b)$ The set $G(\La\times\La^-)$ is also a closed discrete subset of $\m R$.
\end{Cor}

\begin{proof}
This follows from Corollary \ref{corphg}.
\end{proof}

The polynomials $F$ and $G$ satisfy the following equivariant properties
with $\chi$ as in \eqref{eqnchidet}:

\begin{Lem}
\label{lemfgadl}
With the notation \ref{notacthor}. 
For $\ell$ in $L$, $X$ in $\g u$, $Y$ in $\g u^-$, one has
\begin{equation}
\label{eqnfadl}
F({\rm Ad}\ell\, X)=\chi(\ell)^2\, F(X)\, ,
\end{equation}
\begin{equation}
\label{eqngadl}
G({\rm Ad}\ell\, X, {\rm Ad}\ell\, Y)=G(X,Y)\, .
\end{equation}
\end{Lem}

\begin{proof} Note that, for all $\ell$ in $L$, 
one has 
\begin{equation}
\label{eqnchiwo}
\chi(w_0\,\ell\, w_0^{-1})=\chi(\ell)^{-1}\, .
\end{equation}
One computes using \eqref{eqnphlvglu} 
\begin{eqnarray*}
F({\rm Ad}\ell\, X)&=&
\Phi(\ell\, e^X\, w_0\, w_0^{-1}\,\ell^{-1}\, w_0)\\
&=&
\chi(\ell)\Phi(  e^X\, w_0)\chi( w_0^{-1}\,\ell^{-1}\, w_0)
\; =\;\chi(\ell)^2\, F(X)\, .
\end{eqnarray*}
Similarly, one computes
\begin{eqnarray*}
G({\rm Ad}\ell\, X, {\rm Ad}\ell\, Y)
&=&
\Phi(\ell \, e^X\, e^Y\, \ell^{-1})\\
&=&\chi(\ell)\, G(X,Y)\, \chi(\ell)^{-1}
\;=\; G(X,Y)\, .
\end{eqnarray*} 
\end{proof}

In the next proposition, we identify implicitly, using the adjoint action, the group $L$ 
as a subgroup of the group ${\rm Aut}(\g u)$ of automorphisms of $\g u$.
Note that when $G$ is adjoint, the adjoint representation of $L$ on $\g u$ is faithful.

The following proposition and corollary are a converse to Lemma \ref{lemfgadl}.

\begin{Prop}
\label{profphf}
Let $G$ be an adjoint semisimple real algebraic Lie group, let
$U$, $U^{-}$ be opposite horospherical 
subgroups,
and $L:=P\cap P^-$ be the intersection of their normalizers.  
We assume that $U$ is reflexive commutative 
or Heisenberg.\\
$a)$ Let $H:=\{\ph\in {\rm Aut}(\g u)\mid F\circ\ph 
\;\mbox{is proportional to}\; F\}.$ Then the group 
$L$ has finite index in $H$.\\
$b)$ Let $\ell$, $\ell'$ be elements of $L$ such that,
for all $X$ in $\g u$, $Y$ in $\g u^{-}$, one has 
\begin{equation}
\label{eqnglxly}
G({\rm Ad}\ell\, X, {\rm Ad}\ell'\, Y)
= G(X,Y)
\end{equation}
Then one has the equality $\ell=\ell'$.
\end{Prop}

We postpone the purely algebraic 
proof of Proposition \ref{profphf}.a 
in Section \ref{secpolcom} 
and \ref{secpolhei},
splitting it according to the two cases:
when $U$ is reflexive commutative 
and when $U$ is Heisenberg.
The proof of Proposition \ref{profphf}.b will follow 
from the following lemma.

\begin{Lem}
\label{lemfphf}
With the notation as in Proposition \ref{profphf}.
The bilinear form on $\g u\times \g u^-$
given, for 
all $X$ in $\g u$, $Y$ in $\g u^{-}$,  by 
\begin{equation}
\label{eqng2xy}
\textstyle
G_2(X,Y):={\rm tr}_{\g z}
(\pi\, {\rm ad}X\,{\rm ad}Y\,\pi)
\end{equation}
is a
non-degenerate duality between $\g u$ and $\g u^-$.
\end{Lem}

The proof of Lemma \ref{lemfphf} will also be 
postponed  
in Section \ref{secpolcom} 
and \ref{secpolhei}.

\begin{proof}[Proof of Proposition \ref{profphf}.b using Lemma \ref{lemfphf}]
One has the equality
\begin{eqnarray*}
G(X,Y)&=&{\rm det}_\g z(1+
\pi\, {\rm ad}X\,{\rm ad}Y\,\pi +O(\|X\|\,\|Y\|))\\
&=& 
1+{\rm tr_\g z}
(\pi\, {\rm ad}X\,{\rm ad}Y\,\pi) +O(\|X\|\,\|Y\|)\, .
\end{eqnarray*}
Extracting 
the homogeneous component of  degree $2$ in \eqref{eqnglxly}, one gets
$$G_2({\rm Ad}\ell\, X, {\rm Ad}\ell'\, Y)
= G_2(X,Y)
\;\;\mbox{\rm for all $X$ in $\g u$, $Y$ in $\g u^{-}$.}
$$
Since, by Proposition \ref{profphf},
this bilinear form $G_2$ is an $L$-invariant non-dege\-nerate duality, this implies ${\rm Ad}\ell={\rm Ad}\ell'$
on $\g u$ and therefore $\ell=\ell'$.
\end{proof}

%35
\subsection{The single and double  orbits are closed}
\label{secsindou}
\bq
In this section we prove that 
the single and the double $L$-orbits are closed.
Note that this fact does not use the
higher rank assumption on $G$ 
nor the irreducibility assumption on $\Delta$.
\eq

Recall that the {\it single} orbit is the $L$-orbit for the adjoint action of the group $L$ on the space 
$\mc X(\g u)$ of lattices in $\g u$, 
and the {\it double} orbit is the $L$-orbit
on the product space $\mc X(\g u)\times \mc X(\g u^-)$.

\begin{Prop}
\label{prosindou}
Let $G$ be an adjoint semisimple real algebraic Lie group, 
$U$, $U^{-}$ be opposite horospherical 
subgroups,
and $L:=P\cap P^-$ be the intersection of their normalizers.
Let $\Lambda$, $\Lambda^{-}$ be Lie lattices of $\g u$ and $\g u^{-}$ such that the subgroup $\Gamma$ of $G$ generated by $\exp(\Lambda)$ and $\exp(\Lambda^-)$
is discrete.
We assume that $\g u$ is either reflexive commutative or is Heisenberg. Then,
\\
$a)$ the {\it single} $L$-orbit $L\, \Lambda$ is closed in the space $\mc X(\g u)$ of lattices of $\g u$,\\
$b)$ the {\it double} $L$-orbit $L\, (\Lambda,\Lambda^{-})$ is closed in the product 
$\mc X(\g u)\times \mc X(\g u^{-})$.
\end{Prop}

\begin{proof} $a)$ 
Let $\ell_n\in L$ be such that the sequence 
of  lattices ${\rm Ad}\ell_n(\La)$
converges to a lattice $\La_\infty$ of $\g u$. 
We want to find $\ell$ in $L$ such that 
$\La_\infty ={\rm Ad}\ell(\La)$. 

The sequence of covolume of these lattices also converges.
Since the character $\chi$ introduced in
\eqref{eqnchidet}
and the character
$\ell \mapsto {\rm det}_{\rm \g u}({\rm Ad}\ell)$ are proportional, the sequence $\chi(\ell_n)$ also converges.
Therefore, without loss of generality, we can assume that
$$
\chi(\ell_n)=1\, .
$$ 
By Lemma \ref{lemactlat}, there exists 
a sequence of automorphisms $\ph_n\in{\rm Aut}(\g u)$ converging to $e$ such that
\begin{equation*}
{\rm Ad}\ell_n (\La)=\ph_n(\La_\infty)
\;\;\;\mbox{\rm for all $n\geq 1\, $.}
\end{equation*}
For every $X$ in $\La_\infty$, there exists $X_n$ in $\La$
such that
\begin{equation*}
{\rm Ad}\ell_n \, X_n = \ph_n(X)
\;\;\;\mbox{\rm for all $n\geq 1\,$.}
\end{equation*}
By Lemma \ref{lemfgadl}, the elements $\ell_n$ preserve the function $F$ and the sequence 
$$
F(X_n)=F({\rm Ad}\ell_n X_n)= F(\ph_n(X))
\;\;\;\mbox{\rm converges to}\;\;
F(X)\, .
$$
This sequence $F(X_n)$ belongs to the set $F(\Lambda)$
which is closed and discrete by Corollary \ref{corpolfg}.
Therefore this sequence is constant.
Hence, for all $X$ in $\La_\infty$,
there exists an integer $n_X\geq 1$ such that
\begin{equation*}
F(\ph_n(X))= F(X)
\;\;\;\mbox{\rm for all $n\geq n_X\,$.}
\end{equation*}
Since the degrees of the polynomials $F\circ\ph_n-F$ are uniformly bounded,
and since the set $\La_\infty$ is Zariski dense in $\g u$, one can find $n_X$ independent of $X$, that is, there exists 
$n_0\geq 1$ such that
\begin{equation*}
F(\ph_n(X))= F(X)
\;\;\;\mbox{\rm for all $n\geq n_0\,$ and all $X\in \g u$.}
\end{equation*}
Therefore, by Proposition \ref{profphf}.a, the automorphism $\ph_n$ belongs to $L$ for $n$ large.
Hence the single orbit $L\,\La$ is closed.

$b)$ 
The proof is similar to the proof of $a)$. 
Let $\ell_n$ be a sequence in $L$ 
such that the limits 
$$
\La_\infty= \lim_{n\ra\infty}{\rm Ad}\ell_n(\La)
\;\;{\rm and}\;\;
\La^-_\infty= \lim_{n\ra\infty}{\rm Ad}\ell_n(\La^-)
$$
exist.
By $a)$, there exist 
sequences $\eps_n$ and $\eps'_n$ in $L$ converging to $e$ such that
\begin{equation*}
{\rm Ad}\ell_n (\La)={\rm Ad}\eps_n(\La_\infty)
\;\;\mbox{\rm and}\;\;
{\rm Ad}\ell_n (\La^-)={\rm Ad}\eps'_n(\La^-_\infty)\, .
\end{equation*}
For  $(X,Y)$ in $\La_\infty\times\La^-_\infty$,
there exists a sequence $(X_n,Y_n)$ in $\La\times \La^-$
such that 
\begin{equation*}
{\rm Ad}\ell_n \, X_n={\rm Ad}\eps_n\, X
\;\;\mbox{\rm and}\;\;
{\rm Ad}\ell_n \, Y_n={\rm Ad}\eps'_n\, Y\, .
\end{equation*}
By Lemma \ref{lemfgadl}, the elements $\ell_n$ preserve the function $G$ and the sequence 
$$
G(X_n,Y_n)=G({\rm Ad}\ell_n X_n,{\rm Ad}\ell_n Y_n)= 
G({\rm Ad}\eps_n X_n,{\rm Ad}\eps'_n Y_n)
$$
converges to $G(X,Y)$.
Since this sequence $G(X_n,Y_n)$ belongs to the set $G(\Lambda\times\Lambda^-)$
which is closed and discrete by Corollary \ref{corpolfg},
this sequence is constant. 
For all $(X,Y)$ in $\La_\infty\times\Lambda^-_\infty$,
there exists  $n_{X,Y}\geq 1$ such that
\begin{equation*}
G({\rm Ad}\eps_n X_n,{\rm Ad}\eps'_n Y_n)= G(X,Y)
\;\;\;\mbox{\rm for all $n\geq n_{X,Y}\,$.}
\end{equation*}
As before, these $n_{X,Y}$ can be chosen independently of $(X,Y)$.
Therefore, by Proposition \ref{profphf}.b, 
one has $\eps_n=\eps'_n$ for $n$ large.
Hence the double orbit $L\,(\La,\La^-)$ is closed.
\end{proof}

%4
\section{Arithmeticity of $\Gamma$}
\label{secarigam}

The aim of this Chapter is to deduce our main theorem
\ref{thmmar} 
for horospherical subgroups $U$ which are either reflexive commutative or Heisenberg
using the closedness of 
the double $L$-orbit proven in Proposition \ref{prosindou}.
We will prove:

\begin{Prop}
\label{procomhor}
Let $G$ be a semisimple real algebraic Lie group
of real rank at least two, $P$ be a non-trivial parabolic subgroup, 
$U$ be its unipotent radical, 
and $\Gamma$ be a Zariski dense discrete subgroup 
of $G$ that contains an irreducible lattice $\Delta$ of $U$.
Assume  $U$ is reflexive commutative or $U$ is Heisenberg.

Then $\Gamma$ is an irreducible arithmetic lattice of $G$.
\end{Prop}

\begin{Rem}
\label{remcomhor}
Without loss of generality, we can assume $G$ adjoint. 
By Lemma \ref{lemgamuum}, the group $G$ has no compact factors 
and  
$\Gamma$ is a discrete subgroup that contains $\exp(\Lambda)$ and $\exp(\Lambda^-)$
where $\Lambda$, $\Lambda^{-}$ are irreducible lattices in opposite 
horospherical Lie subalgebras $\g u$ and $\g u^{-}$. 
Let $L:=P\cap P^-$ be the intersection of their  normalizers. 
By Proposition \ref{prosindou}, we know that  
the  double  $L$-orbit $L\, (\Lambda,\Lambda^{-})$ is closed.
The proof of Proposition \ref{procomhor} will then be divided into three cases 
which are studied separately in
Sections \ref{secnotmin}, \ref{secmincom} and  \ref{secminhei}. 
\end{Rem}

%41 
\subsection{Irreducible horospherical lattices}
\label{secirrhor}
\bq
We first discuss the definition \ref{defirrlat} of irreducible and 
indecomposable lattices of $U$ 
\eq

The following lemma tells us that for a lattice 
$\Delta\subset U$ that is included in a discrete
Zariski dense subgroup of $G$, 
being {\it irreducible} is equivalent to being
{\it indecomposable}.
 
\begin{Lem}
\label{lemirrhor}
Let $G$ be a semisimple real algebraic Lie group, 
$U\subset G$ a non-trivial horospherical subgroup, 
and $\Delta\subset U$ a lattice of $U$ which is contained in a discrete Zariski dense
subgroup $\Gamma$ of $G$. Then the following are equivalent:\\ 
$i)$ $\Gamma$ is irreducible.\\
$ii)$ $\De$ is irreducible.\\
$iii)$ $\De$ is indecomposable.
\end{Lem}

\begin{Rem}
\label{remirrhor}
Let $\Lambda$ be a Lie lattice of $\g u$ such that ${\rm exp}(\Lambda)$ 
is contained in $\Delta$. 
The irreducibility of $\Delta$ is equivalent to the irreducibility of $\Lambda$:\\
$ii)'$ For any proper ideal $\g g'$ of $\g g$ , one has $\Lambda\cap\g g' = \{0\}$.\\
$iii)'$ Equivalently, there does not exists a decomposition $\g g=\g g'\oplus \g g''$ as a sum of proper ideals
such that $(\Lambda\cap\g g')\oplus(\Lambda\cap\g g'')$ has finite index in $\Lambda$.

In this case the lattice  $\Lambda$ is also called irreducible.
\end{Rem}

\begin{proof}[Proof of Lemma \ref{lemphg}]
Without loss of generality, we can assume $G$ adjoint. 

Since  $\Delta$ is infinite, one has 
$i)\Longrightarrow ii)\Longrightarrow iii)$.

Assume that Condition i) is not satisfied. Choose a 
non-trivial normal subgroup $G'$ of $G$ of minimal dimension such that the group $\Gamma':=\Gamma \cap G' $ is non-trivial.
We can decompose $G$ as a product $G=G' G''$ where $G''$ is 
the centralizer of $G'$. We set
$\Delta':=\Delta \cap G'$ and $\Delta'':=\Delta \cap G''$ and we want to prove that 
the group $\Delta'\Delta''$ has finite index in $\Delta$. 

Let $p':G\ra G'$   be the projection on the first factor.
Since the group $p'(\Gamma)$ is Zariski dense in $G'$ and normalizes the  group $\Ga'$,
the Zariski closure $H'$ of $\Ga'$ is a normal subgroup of $G'$. Therefore, by minimality of $G'$,
one has $H'=G'$ and the group $\Gamma'$ is Zariski dense in $G'$. Since the group $p'(\Gamma)$ normalizes
$\Gamma'$, this group $p'(\Gamma)$ is also a discrete subgroup of $G'$. 

We decompose the horospherical group $U$ as a  product of normal subgroups
$U=U'U''$ where $U':=U\cap G'$ and $U'':=U\cap G''$.
We have just seen that the group $p'(\Delta)$ is a discrete subgroup of $U'$.
Since the group $\Delta$ is a lattice of $U$, this implies that 
the intersection $\Delta''=\Delta\cap U''$ is a lattice of $U''$.

By the same argument as above  the Zariski closure of  
the group  $\Ga'':=\Ga\cap G''$
is a normal subgroup $H''$ of $G''$
that contains $U''$. Let $p'':G\ra H''$ be the projection on the factor $H''$.
By the same argument as above the group $p''(\Ga)$ 
is a discrete subgroup of $H''$. In particular $p''(\Delta)$ 
is a discrete subgroup of $U''$ and, as above, the intersection $\Delta'=\Delta\cap U'$ is a lattice of $U'$.

Therefore the group $\Delta'\Delta''$ is a lattice in $U$ and has finite index in $\Delta$.
\end{proof}

The following lemma tells us that the existence of an irreducible lattice in $U$ imposes compatibility conditions on the factors $U_a$.

\begin{Lem}
\label{lemirruua}
Let $G$ be an adjoint semisimple real algebraic Lie group and
$U\subset G$ be a non-trivial horospherical subgroup.
Let $G_a$ be the simple factors of $G$ and $U_a:=U\cap G_a$.
Assume that $U$ contains an irreducible lattice $\Delta$.
 Then\\ 
$a)$ If one group $U_a$ is commutative, all the groups $U_a$  are commutative.\\
$b)$ If one group $U_a$ is $s$-step nilpotent, all the groups
$U_a$ are $s$-step nilpotent.
\end{Lem}

\begin{proof}
The group $U$ is $s$-step nilpotent for some integer $s\geq 1$. 
This means that the $s$-th term $C^sU$ of the central descending sequence of $U$
is the last non-trivial one.
By Corollary \ref{cornillat}.b, the subgroup $C^s\Delta$ is a lattice of the group $C^sU$. Hence it is non-trivial. 
By irreducibility of $\De$, the projection of $C^s\Delta$ on each $U_a$ is non-trivial.
Therefore the group $C^sU_a$ is non-trivial.
\end{proof}

%42 
\subsection{$\Gamma$-proper and $\Gamma$-compact subgroups}
\label{secprocom}
\bq
We recall in this section
various lemmas due  to Margulis
in \cite{Margulis74} that focus on the intersection of a discrete subgroup of $G$ 
with closed subgroups of $G$.
We give the proofs which are short.
\eq

Let $H$ be a locally compact second countable group, 
let $\Gamma\subset H$ be a discrete subgroup and 
$H_1\subset H$ be a closed subgroup.

\begin{Def}
\label{defgcogpr}
$H_1$ is said to be $\Ga$-compact if $\Gamma\cap H_1$ is cocompact in $H_1$.\\
$H_1$ is said to be $\Ga$-proper if the injective map $H_1/(\Gamma\cap H_1)\hookrightarrow H/\Ga$ is proper.
\end{Def}

Here is the interpretation: Let $x_0:=\Ga/\Ga$ be the basis point of $H/\Ga$.\\
- The group $H_1$ is   $\Ga$-compact if and only if the $H_1$-orbit $H_1x_0$ is compact.\\
- The group $H_1$ is  $\Ga$-proper if and only if  the $H_1$-orbit $H_1x_0$ is  closed in $H/\Ga$.

Indeed, it is a classical fact that the topology induced on a closed orbit $H_1x_0$ 
coincides with the quotient topology on $H_1/(\Gamma\cap H_1)$

\begin{Lem}
\label{lemgcogpr}
Let $H$ be a locally compact second countable group, 
let $\Gamma\subset H$ be a discrete subgroup and 
$H_1$, $H_2$  be  two closed subgroups of $H$. \\
$a)$ If $H_1$ and $H_2$ are $\Ga$-proper, then $H_1\cap H_2$ is $\Ga$-proper.\\
$b)$ If $H_1$ is $\Ga$-proper and $H_2$ is $\Ga$-compact, then $H_1\cap H_2$ is $\Ga$-compact.
\end{Lem}

\begin{proof}
a) By assumption, both orbits $H_1x_0$ and $H_2 x_0$ of the basis point $x_0$ are closed.
Since $\Gamma$ is discrete, all the $H_1\!\cap\! H_2$-orbits in the intersection  
$H_1x_0 \cap H_2 x_0$  are open in this intersection. Therefore these orbits are closed.
In particular, the orbit $(H_1\cap H_2)x_0$ is closed in $H/\Ga$.

b) Since  $H_2x_0$ is compact, the orbit $(H_1\cap H_2)x_0$ is also compact.
\end{proof}

\begin{Lem}
\label{lempglgug}
Let $G$ be a semisimple real algebraic Lie group, 
$P$ and $P^{-}$ be opposite parabolic subgroups,
$U$ and $U^{-}$ their unipotent radicals, 
and $L:=P\cap P^{-}$ so that 
$P=LU$ and  $P^{-}=LU^{-}$.
Let $P_0:=\{ p\in P\mid \det_\g u{\rm Ad}p=1\}$
and $L_0:=L\cap P_0$.
Let $\Gamma$ be a discrete subgroup of $G$.\\
$a)$ If $U$ is $\Gamma$-compact, the group $P_0$ is $\Gamma$-proper.\\
$b)$ If moreover $U^-$ is $\Ga$-compact, the group $L_0$
is also $\Ga$-proper.\\
$c)$ In this case, the group $(\Gamma\cap L)(\Ga\cap U)$
has finite index in $\Ga\cap P$.
\end{Lem}

\begin{Rem}
The group $P_0$ is called the {\it unimodular normalizer}
of $U$. Since  $\Ga\cap P$ normalizes the lattice $\De:=\Ga\cap U\subset U$,
one has   $\Ga\cap P\subset P_0$.
\end{Rem}

\begin{proof}
$a)$ Let $p_n\in P_0$ and $\ga_n\in \Ga$ be sequences such that
the sequence 
$ 
g_n:= p_n\ga_n$
converges to $g_\infty\in G$.
We want to find a sequence $\de_n\in \Ga\cap P_0$ 
such that the sequence 
$
p_n\delta_n$ has a limit point $p_\infty\in P_0.$
Since the sequence $g_n$ converges, 
there exists a neighborhood $\Omega_0$ of $e$ in $G$
such that the groups 
$\Delta'_n:=p_n\Delta p_n^{-1}
\subset g_n\Ga g_n^{-1}$ intersect $\Om_0$ trivially.
Since these lattices of $U$ have the same covolume,
by the Mahler compactness criterion, after extraction, 
these lattices $\De'_n$ converge  to a lattice $\De'_\infty$ of $U$. 
Therefore the subgroups $\De_n=g_n^{-1}\De'_n g_n=
\ga_n^{-1}\De\ga_n$
of $\Ga$ 
converge to the subgroup
$\De_\infty:=g_\infty^{-1}\De'_\infty g_\infty$.

Since $\Ga$ is discrete and $\De_\infty$
is finitely generated, there exists $n_0\geq 1$ such that
$
\De_n=\De_\infty
\;\;\mbox{\rm for all $n\geq n_0$.}
$
Therefore the elements $\de_n:=\ga_n\ga_{n_0}^{-1}$
belong to $\Gamma\cap P_0$ and
the sequence $p_n\de_n$ converges to $p_\infty:=g_\infty\ga_{n_0}^{-1}$.

$b)$ By point $a)$, both groups $P_0$ and $P_0^-$ are
$\Ga$-proper. Therefore, by Lemma \ref{lemgcogpr}.a,
the group $L_0=P_0\cap P_0^{-}$ is also $\Ga$-proper.

$c)$ 
Assume by contradiction that the group
$(\Gamma\cap L_0)(\Ga\cap U)$
has infinite index in $\Ga\cap P_0$.
Let 
$
\ga_n=u_n\ell_n
\;\;{\rm with}\;\;
\ell_n\in L_0\;, \;\; u_n\in U
$ 
be a sequence of elements of $\Ga\cap P_0$
whose images in $(\Ga\cap U)\backslash G/(\Gamma\cap L_0)$ are distinct.
After multiplying on the left 
by elements of the lattice $\Delta\subset U$, 
one can assume that the sequence $u_n$ is bounded.
Then the sequence 
$
\ell_n x_0=u_n^{-1}x_0$
is also bounded in $G/\Ga$.
Since  $L_0$ is $\Ga$-proper, 
one can find a sequence $\de_n$ in $\Gamma\cap L_0$ 
such that the sequence $m_n:=\ell_n\de_n$ is bounded in $L_0$.
Then the sequence $\ga_n\de_n=u_nm_n\in \Ga\cap P_0$ 
is also bounded. Since $\Ga$ is discrete it can only take finitely many values.
Contradiction.
\end{proof}

%43 
\subsection{The Raghunathan-Venkataramana theorem}
\label{secventhe}
\bq
The following result of Raghunathan and Venkataramana in \cite{Raghunathan76}
and \cite{Venkataramana94}
says that in order to prove Theorem \ref{thmmar} 
for $U$ refle\-xive, we only have to find a $\m Q$-structure on $G$ such that $\Gamma\subset G_\m Q$.
\eq

\begin{Prop}
\label{proragven}
Let $G$ be a semisimple  real algebraic Lie group
of real rank at least two. Assume that $G$ is defined over $\m Q$ and is $\m Q$-simple.
Let $U$ and $U^{-}$ be opposite horospherical subgroups
which are defined over $\m Q$
and $\Delta\subset U_\m Z$ and $\Delta^-\subset U^-_\m Z$ be finite index subgroups.
Then the subgroup $\Gamma$ generated by $\Delta$ and 
$\Delta^-$ has finite index in $G_\m Z$.
\end{Prop}

\begin{proof}
This proposition is proven in \cite{Raghunathan76}
and \cite{Venkataramana94} for maximal opposite horospherical $\m Q$-subgroups. 
But Hee Oh explains in \cite[Corollary 2.2.2]{Oh98a} how to deduce the
general case. Indeed $U$ and $U^-$
are normal subgroups in maximal opposite horospherical $\m Q$-subgroups
$V$ and $V^-$. 
Let $L$ be the intersection of the normalizers of 
$U$ and $U^-$ so that $V=(L\cap V)U$ and $V^-=(L\cap V^-)U^-$.
There  exists a finite index subgroup 
$\Delta_L\subset L_\m Z$ that normalizes both $\Delta$ and $\Delta^-$.
By \cite{Raghunathan76}
and \cite{Venkataramana94}, the group generated by 
the lattices $(\Delta_L\cap V)\Delta$ and 
$(\Delta_L\cap V^-)\Delta^-$ of $V_\m Z$ and $V^-_\m Z$ 
has finite index in $G_\m Z$
and normalizes $\Gamma$. Therefore, by the Margulis normal subgroup
theorem, $\Gamma$ has finite index in $G_\m Z$.
\end{proof}

%44
\subsection{Margulis' extension of $\m Q$-forms}
\label{secextfor}

\bq
In this section we introduce an important tool which allows us to extend a pair of 
compatible $\m Q$-forms on 
$\g u$ and $\g u^-$ to a $\m Q$-form on $\g g$.
This tool (Corollary \ref{corextfor}) tells us that in order to prove Theorem \ref{thmmar}
for a reflexive horospherical subgroup $U$, we only have to find elements of $L$ that normalize
both lattices $\La$ and $\La^-$.
\eq

\begin{Prop}
\label{proextfor}
Let $G$ be a semisimple adjoint real algebraic Lie group,
$P$, $P^-$  be two opposite parabolic subgroups,
$U$, $U^-$ be their unipotent radicals and $L=P\cap P^-$.
Let $\Gamma$ be a discrete Zariski dense subgroup of $G$ 
such that $\Ga\cap U$ is an irreducible lattice of $U$
and $\Ga\cap U^-$ is an irreducible lattice of $U^-$.  
Assume that
the group $\Ga\cap L$ is infinite.

Then there exists a $\m Q$-form $G_\m Q$ of $G$ such that $\Gamma\subset G_\m Q$
\end{Prop}

This proposition  is an extension of a result of Margulis
\cite[Theorem 9.5.1]{Margulis74} where it is assumed that 
$\Ga$ is a lattice in $G$.
For the sake of completness we carefully check below the proof 
of this extension.
Notice that in this proposition there is no higher rank assumption.
The proof relies on the following  lemma
of Margulis \cite[Lemma 8.6.2]{Margulis74}.

\begin{Lem}
\label{lemextfor}
Let $V$ be a real vector space, and $W, W^-\subset V$ be two real vector subspaces. 
Let $G\subset {\rm GL}(V )$ be a Zariski connected algebraic subgroup,
and $\Omega_0\subset G$ be a Zariski open subset such that:\\ 
$i)$ The dimension $d:=\dim(W^-\cap gW )$ does not depend on $g\in   \Omega_0$.\\
$ii)$ The vector subspaces $gW$ for $g$ in $\Omega_0$ span $V$.\\
$ii)'$ The vector subspaces $g^{-1}W^-$ for $g$ in $\Omega_0$ span $V$.\\
$iii)$ The vector subspaces $W^-\cap gW$ for $g$ in $\Omega_0$ span $W^-$.\\
$iii)'$ The vector subspaces $g^{-1}W^-\cap W$ for $g$ in $\Omega_0$ span $W$.

Let $\Ga\subset G$ be a Zariski dense subgroup. 
We assume that $W$ and $W^-$ are endowed with  $\m Q$-forms $W_\m Q$ 
and $W^-_\m Q$
satisfying the following compatibility conditions:\\
$iv)$ For all $g$ in $\Gamma\cap \Omega_0$ the subspaces $W^-\cap gW$ of $W^-$ are defined over $\m Q$.\\
$iv)'$ For all $g$ in $\Gamma\cap \Omega_0$ the subspaces $g^{-1}W^-\cap W$ of $W$ are defined over $\m Q$.

Then there exists a $\m Q$-form $G_\m Q$ on $G$ such that $\Gamma\subset G_\m Q$.
\end{Lem}

\begin{proof}[Proof of Proposition \ref{proextfor}]  
We apply Lemma  \ref{lemextfor} with $V:=\g g$ 
and with $\Omega_0:=U^-P$.
The vector subspaces $W$ and $W^-$ will be defined later.  
Let $Q$, $Q^-$ and $H$ be the Zariski closures of $\Gamma\cap P$,
$\Gamma\cap P^-$ and $\Gamma\cap L$ respectively,
and let $\g q$, $\g q^-$ and $\g h$  be  the Lie algebras of  $Q$, $Q^-$ and $H$. 
\vs 
 
{\bf First step}:
we  describe a $\m Q$-form on $\g q$ and $\g q^-$.\\
The group $\Delta:=\Gamma\cap U$ is a lattice of $U$.
Hence, by Lemma \ref{lemnillat}.a, 
there exists a unique $\m Q$-form $U_\m Q$ on $U$ 
such that $\Delta\subset U_\m Q$.
The group ${\rm Aut}(U_\m Q)\ltimes U_\m Q$ is then 
a $\m Q$-form of the algebraic group 
${\rm Aut}(U)\ltimes U$. 
The group $P=L\ltimes U$ is naturally an algebraic  subgroup of 
${\rm Aut}(U)\ltimes U$. This subgroup might not be defined over $\m Q$. 
However the group $\Gamma\cap P$ is included
in ${\rm Aut}(U_\m Q)\ltimes U_\m Q$ 
and therefore its Zariski closure $Q$ 
is defined over $\m Q$. Let $\g q_\m Q$ 
be the corresponding $\m Q$-form of the Lie algebra $\g q$
of $Q$. This $\m Q$-form is $\Gamma\cap P$-invariant.
We use a similar construction for the $\m Q$-form on $\g q^-$.
\vs

{\bf Second step}:
we study the intersections $\g q^-\cap {\rm Ad}g(\g q)$. \\
For $g$ in $\Gamma\cap \Omega_0$,
the parabolic subgroups 
$P^-$ and $gPg^{-1}$ are opposite
and their intersection $L_{g}:=P^-\cap gPg^{-1}$
is a reductive group. Let $\g l_g$ be its Lie algebra.
The Zariski-closure 
$H_{g}$ of the group $\Gamma\cap L_g$ is 
a subgroup of $Q^-$ which is defined over $\m Q$. 
Hence its Lie algebra $\g h_{g}$ is 
a subalgebra of $\g q^-$ which is defined over $\m Q$.
Since the groups $U^-$ and $gUg^{-1}$ are $\Ga$-compact, 
by Lemma \ref{lempglgug}, the group
$$
\mbox{
$(\Gamma\cap L_g)(\Gamma\cap U^-)$ has finite index in $\Gamma\cap P^-$.}
$$
Therefore, the group $H_gU^-$ has finite index in $Q^-$ and one has
$\g q^-=\g h_{g}\oplus \g u^-\, .$ 
By a similar argument with the element 
$g^{-1}$, one gets
$
\g q={\rm Ad} g^{-1}(\g h_{g})\, \oplus\, \g u.$
Combining these last two equalities, one gets
\begin{equation}
\label{eqnhggqq}
 \g h_{g}= \g q^-\cap {\rm Ad}g(\g q)\, .
\end{equation}
This proves that the dimension of this intersection 
is equal to $\dim\g q -\dim \g u$, hence does not depend on $g$ and that this intersection  is a vector subspace of $\g q^-$ which is defined over $\m Q$. 
\vs
  
{\bf Third step}: we apply Lemma \ref{lemextfor}.

Let 
$W^-\subset \g q^-$ be the vector space spanned by  $\g h_g$
for $g$ in $\Ga\cap \Omega_0$,
and $W\subset \g q$ be the vector subspace spanned by
${\rm Ad} g^{-1}(\g h_{g})$, for  $g$ in $\Gamma\cap \Omega_0$.
The subspace $W^-$ of $\g q^-$ is defined over $\m Q$.
Similarly, the subspace $W$ of $\g q$ is defined over $\m Q$. We can now check the assumptions of 
Lemma  \ref{lemextfor}. 

$i)$ and $iv)$ By \eqref{eqnhggqq} one still has the equality
$$
\g h_{g}= W^-\cap {\rm Ad}g(W)
\;\;\mbox{\rm for all $g$ in $\Gamma\cap \Omega_0$}\, .
$$ 
  
$ii)$ By Lemma \ref{lemirrhor}, the group $\Ga\cap L$ 
intersects trivially all proper normal subgroup $G'$ of $G$.
Therefore, for any projection $p_a:\g g\ra\g g_a$
onto a simple ideal, the Lie algebra $p_a(\g h)$ is non-zero.
The subspace of $\g g$ spanned by ${\rm Ad}g(W)$, for $g\in \Gamma\cap \Om_0$,
is an ideal of $\g g$ that contains $\g h$. Therefore it is equal to $\g g$.

$iii)$  follows directly from our choice of $W^-$. 

$ii)'$, $iii)'$ and $iv)'$ are proven in a similar way.

By Lemma \ref{lemextfor}, there is a
$\m Q$-form on $G$ such that $\Gamma\subset G_\m Q$.
\end{proof}

\begin{Cor}
\label{corextfor}
Let $G$ be an adjoint semisimple real algebraic Lie group
of real rank at least two. 
Let $\Gamma$ be a discrete subgroup 
that contains $\exp(\Lambda)$ and $\exp(\Lambda^-)$
where $\Lambda$, $\Lambda^{-}$ are irreducible lattices in opposite 
horospherical Lie subalgebras $\g u$ and $\g u^{-}$. 
Let $L:=P\cap P^-$ be the intersection of their  normalizers. 
Assume that the stabilizer $L_{\La,\La^-}$ in $L$ of the pair $(\La,\La^-)$ is infinite.

Then $\Ga$ is an irreducible arithmetic lattice of $G$. 
\end{Cor}

\begin{proof}
Without loss of generality, 
we can assume that $\Gamma$ is generated by $ \exp(\Lambda)$ and $\exp(\Lambda^-)$.
Let $\Ga'$ be the normalizer of $\Ga$ in $G$. 
Since the group $\Ga$ is discrete and Zariski dense in $G$,
the group  $\Ga'$ is also discrete and Zariski dense in $G$.
Since this group $\Ga'$ contains  $\exp(\Lambda)$, $\exp(\Lambda^-)$ and 
the infinite group $L_{\La,\La^-}$,
by Proposition \ref{proextfor}, there exists a $\m Q$-form $G_\m Q$
of $G$ such that $\Ga'\subset G_\m Q$.
In particular, one has $\Ga\subset G_\m Q$.
Since the lattice $\Lambda$ is irreducible this $\m Q$-form is $\m Q$-simple.
By Raghunathan and Venkataramana's Proposition \ref{proragven},
$\Ga$ is commensurable to the lattice $G_\m Z$.  
\end{proof}

%45
\subsection{When $P$ is not a minimal parabolic subgroup}
\label{secnotmin}

\bq
We give the proof of Proposition \ref{procomhor} in the next three sections.
In almost all cases we will   prove that
the stabilizer
$ L_{\La,\La^-}$ is infinite and apply Corollary \ref{corextfor}.

We first prove Proposition \ref{procomhor} when the parabolic $P$ is not minimal, i.e. when the semisimple 
group $[L,L]$ is non-compact.
\eq

We use the notation of Remark \ref{remcomhor}. Let $S$ be the product of the non-compact simple factors of $[L,L]$
and let $H\subset L$ be the closure of the product $L_{\La,\La^-}S$.
Since $S$ is a normal subgroup of $L$, $H$ is a closed subgroup of $L$. 
By Proposition \ref{prosindou},  the  $L$-orbit $L\,(\La,\La^-)$ is closed.
Hence the  closure  of the $S$-orbit $S\,(\La,\La^-)$ is the orbit $H\,(\La,\La^-)$.

Note that the Radon $H$-invariant measure $\la_0$ on this orbit 
$H\,(\La,\La^-)$ is $S$-ergodic. 
Applying the following Proposition \ref{prodanmar} to the finite volume homogeneous space $\mc X_0:=\mc X(\g u)\times \mc X(\g u^-)$ tells us that $\la_0$ has finite volume.
This means that the stabilizer $L_{\La,\La^-}$
is a lattice in $H$.
Since $H$ is non-compact this stabilizer
$ L_{\La,\La^-}$ is infinite and by Corollary \ref{corextfor},
the group $\Ga$ is an irreducible arithmetic lattice of $G$.
This proves Proposition \ref{procomhor} when  $P$ is not minimal.
\vs

We have used the following result  due to Dani and Margulis. 

\begin{Prop}
\label{prodanmar}
Let $G_0$ be a real algebraic Lie group,  $\Ga_0$ be a lattice in $G_0$ and $\mc X_0=G_0/\Ga_0$.
Let $S_0\subset G_0$ be a semisimple subgroup 
then any $S_0$-invariant and ergodic Radon measure $\la_0$ on $\mc X_0$ has finite volume. 
\end{Prop} 

\begin{proof}
See \cite[Theorem 11.5]{MargulisTomanov94}. 
Write $S_0=S'_0S''_0$ where $S'_0$ and $S''_0$ are respectively the product of the compact and non-compact factors of $S_0$. 
Let $U_0$ be a one parameter unipotent subgroup of $S''_0$
which is not included in a proper normal subgroup of $S''_0$.
The Dani-Margulis uniform recurrence theorem 
for unipotent flow on finite volume homogeneous spaces implies that the Hilbert space
$L^2(\mc X_0,\la_0)$ contains non-zero $U_0$-invariant functions. Therefore the 
Howe-Moore theorem gives a
non-zero $S''_0$-invariant function $\ph\in L^2(\mc X_0,\la_0)$.
One can assume $1\leq \ph\leq 2$. 
Averaging $S'_0$-translates of $\ph$
gives a non-zero $S_0$-invariant function. 
By ergodicity this function must be almost surely constant and hence $\la_0$ has finite volume. 
\end{proof}

%46 
\subsection{When $P$ is minimal and $U$ is commutative}
\label{secmincom}
\bq
In this section we prove Proposition \ref{procomhor} when the parabolic $P$ is  minimal and 
$U$ is  commutative.
\eq

In this case the group $G$ is a product  
\begin{equation}
\label{eqnsod}
\textstyle
G=\prod_a G_a
\;\;{\rm where}\;\; 
G_a={\rm PSO}(d_a+1,1)  
\;\;{\rm with}\;\; 
d_a\geq 1\, .
\end{equation}
This case is due to the second author in \cite{Miquel17}.
We present a short proof below using our new approach.
Let 
$L_0:=\{\ell\in L\mid {\rm det}_\g u({\rm Ad}\,\ell)=1\}$. 

\begin{Lem}
\label{lemsod}
With the notation 
of Remark $\ref{remcomhor}$. We assume that $P$ is minimal 
and $U$ is commutative. Then 
the $L_0$-orbit $L_0 (\La,\La^-)$ is compact
\end{Lem}

\begin{proof} By Proposition \ref{prosindou},  the  $L_0$-orbit $L_0(\La,\La^-)$ is closed.
We only need to check that this orbit is relatively compact. By \eqref{eqnsod}, 
$\g u$ is a direct sum
\begin{equation*}
\label{eqnuuarda}
\textstyle
\g u=\oplus_a\,\g u_a
\;\;{\rm where}\;\; 
\g u_a=\m R^{d_a}
\end{equation*}
are Euclidean vector spaces and $L$ is the product of groups of similarities 
\begin{equation*}
\label{eqnllasim}
\textstyle
L=\prod_a L_a
\;\;{\rm where}\;\; 
L_a={\rm Sim}(\m R^{d_a} )  
\, .
\end{equation*}
We use again the polynomial $F(X)$ introduced in \eqref{eqnpolfg}.
By Lemma \ref{lemfgcom}.a, $F$ is non zero. 
Since $L$ has an open orbit in $\g u$,
$F$ is the unique $L$-semi-invariant polynomial on $\g u$ with character $\chi^2$ i.e.
satisfying the equivariance property 
\eqref{eqnfadl}. Therefore, there exists $c>0$ such that
$$\textstyle
F(X)=c\,\prod_a \|X_a\|^{2d_a}
\;\;\;\mbox{\rm for all $X=(X_a)\in \g u$.}
$$

Since the lattice $\La$ is irreducible,  one has 
$$
F(X)\neq 0
\;\;\mbox{\rm for all $X$ in $\La\!\smallsetminus\!\{0\}$\, .}
$$    
Lemma \ref{corpolfg} tells us that the set $F(\La)$ is a closed discrete subset of $\m R$. Therefore 
the following constant is positive
$$
m:=\inf_{X\in \La\smallsetminus\{0\}}F(X) >0\, .
$$
Since the function $F$ is $L_0$-invariant, 
one computes, for all $\ell$ in $L_0$,
\begin{equation}
\label{eqnfadlx}
\inf_{X\in \La\smallsetminus\{0\}}\|{\rm Ad} \ell(X)\|^{2d}
\geq c^{-1}\inf_{X\in \La\smallsetminus\{0\}}F({\rm Ad}\ell(X)) 
=c^{-1}m>0\, ,
\end{equation}
where $d:=\sum_a\, d_a$.
Since all the lattices ${\rm Ad}\ell\,(\Lambda)$,
with $\ell\in L_0$, have the same covolume,
by the Mahler compactness theorem, the bound \eqref{eqnfadlx}
tells us that the orbit 
$L_0\La$ is relatively compact.
For the same reason the orbit $L_0\La^-$ is also relatively compact.
\end{proof}

We can conclude the proof of Proposition \ref{procomhor} in this case.
Since the real rank of $G$ is at least two the group $L_0$ 
is non-compact. Therefore, by Lemma \ref{lemsod}, the stabilizer
$  L_{\La,\La^-}$ is infinite and by Corollary \ref{corextfor},
the group $\Ga$ is an irreducible arithmetic lattice of $G$.

%47 
\subsection{When $P$ is minimal and $U$ is Heisenberg}
\label{secminhei}

\bq
In this section we prove Proposition \ref{procomhor} when the parabolic $P$ is  minimal and 
$U$ is  Heisenberg.
\eq

In this case the group 
$G$ is a simple real Lie group and its root system $\Sigma$ is of type $A_2$. 
This means that the root system is 
$\Si=\{\pm\al_1,\pm\al_2,\pm\al_3\}$ with $\al_3=\al_1+\al_2$
and the Lie algebra $\g u$ is the sum 
$$
\g u=\g g_1\oplus\g g_2
\;\;{\rm with}\;\;
\g g_1:=\g g_{\al_1}\oplus\g g_{\al_2}
\;\;{\rm and}\;\;
\g g_2:=\g g_{\al_3}\, .
$$ 

One can give a list of such groups $G$: they are the groups ${\rm PGL}(3,\m K)$, for the fields $\m K=\m R$, $\m C$, $\m H$, or the group
${\rm Aut}(\m P^2(\m O))$ of collinea\-tions of the projective plane
over the  algebra of octonions. We will not use this list.

The group $L$ is a product $L=MA$ where $M$ is a compact group and $A$ is a two-dimensional split torus.
The group $L_0:=\{\ell\in L\mid {\rm det}_\g z({\rm Ad}\,\ell)=1\}$
is equal to $L_0=MA_0$ where $A_0:=L_0\cap A$  is a one-dimensional split torus.
 
For any root $\al$ we denote by $G_\al$ the unipotent group
whose Lie algebra is the root space $\g g_\al$ and $\Ga_\al:=\Ga\cap G_\al$. All these root spaces have the same dimension $d=1$, $2$, $4$ or $8$.

\begin{Lem}
\label{lemheidic}
With the notation 
of Remark $\ref{remcomhor}$. 
We assume that $P$ is minimal 
and $U$ is Heisenberg. Then one has the dichotomy:\\
$(i)$ either  the $L_0$-orbit $L_0 (\La,\La^-)$ is compact,\\
$(ii)$ or $\Ga_\al$ is a lattice in $G_\al$ for all $\al\in \Sigma$.
\end{Lem}

\begin{proof} 
When $\al$, $\be$ are two non-opposite roots whose sum is not a root,
the product  $G_{\al}G_{\be}$ is a unipotent group. We distinguish 
two cases.
\vs

{\bf First case}: Assume $(\Ga\cap G_{\al_1}G_{\al _3})\subset G_{\al_3}\;$
or $\;(\Ga\cap G_{\al_2}G_{\al _3})\subset G_{\al_3}$.\\
In this case we will prove $(i)$. Assume for instance the second inclusion
\begin{equation}
\label{eqngaggg}
(\Ga\cap G_{\al_2}G_{\al _3})\subset G_{\al_3}\, .
\end{equation}

By Proposition \ref{prosindou},  the double  $L_0$-orbit $L_0 (\La,\La^-)$  is closed.
We only need to check that both single orbits $L_0\La$ and $L_0\La^-$
are relatively compact. 
Let $t\mapsto b_t$ be the one-parameter 
subgroup of $L_0$ whose action on an element 
$X=X_1\!+\!X_2\!+\!X_3\in \g u$ with $X_i\in \g g_{\al_i}$, is given by
\begin{equation}
\label{eqnbtxxx}
{\rm Ad}\,b_t(X_1\!+\!X_2\!+\!X_3)=e^t\,X_1+e^{-t}X_2+X_3
\end{equation} 
By contradiction, assume for instance that the orbit $L_0\La$ is not relatively compact.
Since the orbit $L_0\La$ is closed this implies that 
in both directions $n\ra\pm\infty$ the lattices ${\rm Ad}b_n(\La)$
go to infinity. Since all these lattices have the same covolume, 
by the Mahler compactness criterion, 
for all $n>1$, one can choose a non-zero element 
\begin{equation}
\label{eqnxnxnxn}
X_n=X_{n,1}+X_{n,2}+X_{n,3}\in\La\smallsetminus\{ 0\}
\;\;{\rm with}\;\;
X_{n,i}\in \g g_{\al_i}\, .
\end{equation}
such that 
\begin{equation}
\label{eqnenxenx}
\lim_{n\ra \infty}\; e^nX_{n,1}\!+\!e^{-n}X_{n,2}\!+\!X_{n,3} =0\,. 
\end{equation}
We use again the   polynomial $F(X)={\rm det}_{\g z}(e^{{\rm ad}X}w_0)$
introduced in \eqref{eqnpolfg}. 
This polynomial has degree $4d$ and its 
homogeneous component $F_{4d}$ of degree $4d$ 
is given in \eqref{eqnf4d} by
$$F_{4d}(X_1\!+\!X_2\!+\!X_3)={\rm det}_{\g 
g_{_3}}(\tfrac{1}{24}{\rm ad}(X_1\!+\! X_2)^4w_0).
$$ 
According to Lemma \ref{lemfghei1}.b,   $F_{4d}$ is non zero. Since $L$ has an open orbit in 
$\g g_1$,
this polynomial $F_{4d}$ is the unique $L$-semi-invariant polynomial on $\g g_1$ with character $\chi^2$.
Therefore there exists $c\neq 0$ such that
$$
F_{4d}(X_1\!+\!X_2\!+\!X_3)=c\,\|X_1\|^{2d}\,\|X_2\|^{2d}\, ,
$$
where these norms $\|.\|$  
are $M$-invariant norms on $\g g_{\al_i}$.
By Corollary \ref{corpolfg} the set 
$F(\La)$ is closed and discrete.
Since 
$$
F(X_n)=F({\rm Ad}b_n(X_n))
\;\;\mbox{\rm converges to $0$}\;\;
$$
One gets $F(X_n)=0$ for $n$ large.
The same argument  proves that 
$$
F(pX_n)=0
\;\; \mbox{for $n$ large, for all integer $p\geq 1$.}
$$
Therefore, for $n$ large,  one has 
$F_{4d}(X_n)=0$, i.e. 
\begin{equation}
\label{eqnxnxno}
\|X_{n,1}\|\,\|X_{n,2}\|=0\, .
\end{equation}
Since $\La$ is a lattice, combining \eqref{eqnxnxnxn}, \eqref{eqnenxenx} and \eqref{eqnxnxno},
one gets $X_{n,1}=0$
Therefore, by \eqref{eqngaggg}, one also gets $X_{n,2}=0$.
Using again \eqref{eqnxnxnxn}, \eqref{eqnenxenx} one finally gets $X_{n,3}=0$.
Contradiction.
\vs

{\bf Second case}: Assume 
$(\Ga\cap G_{\al_1}G_{\al _3})\not\subset G_{\al_3}\;$
and $\;(\Ga\cap G_{\al_2}G_{\al _3})\not\subset G_{\al_3}$.\\
We choose  elements 
$g_1\in (\Ga\cap G_{\al_1}G_{\al _3})\smallsetminus \Ga_{\al_3}$ and 
$g_2\in (\Ga\cap G_{\al_2}G_{\al _3})\smallsetminus \Ga_{\al_3}$. 
We will prove $(ii)$. 
By Corollary \ref{cornillat}.c, we already know that  
$\Ga_{\al_3}$ is a lattice in $G_{\al_3}$ and $\Ga_{-\al_3}$ is a lattice in $G_{-\al_3}$.

We claim that {\it $\Ga_{\al_1}$ is a lattice in $G_{\al_1}$
and $\Ga_{-\al_2}$ is a lattice in $G_{-\al_2}$.}

Let $\Ga'$ be the discrete subgroup of $G$ 
generated by $g_1$, $\Ga_{\al _3}$ and $\Ga_{-\al _3}$. 
Let $G'$ be the
semidirect product 
$G'=S'\ltimes U'$ where $S'$ is the  simple Lie group of real rank one
generated by $G_{\al _3}$ and $G_{-\al _3}$
and $U'$ is the  unipotent group $U'=G_{\al_1}G_{-\al_2}$.
Since the action of $S'$ on $\g u'$ is irreducible, 
the Zariski closure of $\Ga'$ is the group $G'$.

By Auslander's Proposition  \ref{proaus}
below, the projection of $\Ga\cap G'$ on $S'$ is discrete.
Since $\Ga_{\al_3}$ is a cocompact lattice in $G_{\al_3}$, 
the group $\Ga_{\al_1}\Ga_{\al_3}$ has finite index in 
$\Ga\cap G_{\al_1}G_{\al _3}$.
The existence of $g_1$ implies then that $\Ga_{\al_1}$
contains an element $g'_1\neq e$.
The inclusion $[g'_1,\Ga_{-\al_3}]\subset \Ga_{-\al_2}$
and the equality $[g'_1,G_{-\al_3}]=G_{-\al_2}$
imply that $\Ga_{-\al_2}$ is a lattice in $G_{-\al_2}$.
The inclusion $[\Ga_{-\al_2},\Ga_{\al_3}]\subset \Ga_{\al_1}$
and the equality $[G_{-\al_2},G_{\al_3}]=G_{\al_1}$ 
imply that $\Ga_{\al_1}$ is a lattice in $G_{\al_1}$.

Replacing $g_1$ by $g_2$, one deduces 
that $\Ga_{\al_2}$ is a lattice in $G_{\al_2}$
and $\Ga_{-\al_1}$ is a lattice in $G_{-\al_1}$ 
which proves $(ii)$.
\end{proof}

In this proof we have used the following 
classical result of Auslander which can be found in \cite[Thm 8.24]{Raghunathan72}.

\begin{Prop}
\label{proaus}
Let $G$ be a real algebraic Lie group which is a semidirect product
$G:=S\ltimes U$ of a semisimple Lie group $S$ and of 
a normal unipotent subgroup $U$. Let $p:G\rightarrow S$ be the projection
and $\Ga$ be a Zariski dense discrete subgroup of $G$.
Then the group $p(\Ga)$ is a discrete subgroup of $S$.
\end{Prop}

We now end the proof of Proposition \ref{procomhor} when $P$ is minimal and $U$ is Heisenberg, distinguishing between both cases
of Lemma \ref{lemheidic}.

In Case $(i)$, since $L_0$ is non-compact, the stabilizer
$  L_{\La,\La^-}$ is infinite and by Corollary \ref{corextfor},
the group $\Ga$ is an irreducible arithmetic lattice of $G$.

In Case $(ii)$, the group $\Ga$ intersects cocompactly the two 
opposite horospherical subgroups
$U'=G_{\al_1}G_{-\al_2}$ and ${U'}^-=G_{-\al_1}G_{\al_2}$.
Let $L':=P'\cap {P'}^-$ be the intersection of their normalizers.
The intersection $\Ga\cap L'$ is also infinite since it contains $\Ga_{\al_3}$.
Therefore by the same  Corollary \ref{corextfor},
the group $\Ga$ is an irreducible arithmetic lattice of $G$.
\vs

This finishes the proof of Proposition \ref{procomhor}.

%5
\section{Reduction steps}
\label{secmaithe}

The aim of this chapter is to prove our main theorem 
\ref{thmmar} for all horospherical groups, relying on the special case when 
the horospherical group $U$ is either reflexive commutative
or Heisenberg (Proposition \ref{procomhor}). The reduction process relies on the following
three steps:  Propositions \ref{prononref}, \ref{proroospa}
and \ref{prononroo}. 
The first one deals with non-reflexive groups $U$,
the second one with reflexive non-Heisenberg $U$ whose center is a root space, and 
the last one with reflexive non-commutative $U$ whose center is not a root space.
For a simple group $G$, this reduction process is due to Hee Oh in \cite{Oh98a}, \cite{Oh98b} and \cite{Oh99}. 
We extend it here to  semisimple groups $G$. We include the proofs for the sake of completeness.

%51 
\subsection{Non-reflexive horospherical subgroups}
\label{secnonref}
\bq
We first explain how to reduce the case of  
non-reflexive horospherical groups to the case of 
reflexive horospherical groups.
\eq

\begin{Prop}
\label{prononref}
Let $G$ be a semisimple real algebraic Lie group,
$U$ be a non-reflexive horospherical subgroup,
and $\Gamma$ be a Zariski dense discrete subgroup 
of $G$ that contains an irreducible lattice $\Delta$ of $U$.

Then there exists a larger reflexive  horospherical subgroup 
$U'$ of $G$ containing $U$ such that the group $\Gamma$
also contains an irreducible lattice $\Delta'$ of $U'$.
\end{Prop}

\begin{proof}
We will find a larger  horospherical subgroup 
$V\supsetneq U$ of $G$ such that the group $\Gamma\cap V$
is also an irreducible lattice of $V$. If $V$ is reflexive we are done. If not, we apply again this construction to $V$
until we get a reflexive horospherical subgroup.

We use the notation of \S \ref{secroosys} 
and \S \ref{secgralie}. We choose a maximal split torus $A$ of $G$, 
a set of simple root $\Pi$ and a subset $\theta\subset \Pi$
such that the Lie algebra
$\g u$ and its normalizer $\g p$
are given by
\begin{equation}
\label{eqnuupp}
\g u=
\textstyle\bigoplus\limits_{n_\th(\al)>0}\g g_\alpha
\;\;\;{\rm and}\;\;\;
\g p=\g l \oplus\g u
\;\;{\rm with}\;\; 
\g l:= \textstyle\bigoplus\limits_{n_\th(\al)= 0}\g g_\alpha
\end{equation}
where $n_\th$ is the function on $\Sigma\cup\{0\}$ given by \eqref{eqnnth}.
Let $P=LU$ be the normalizer of $U$ and $w_0$ be the longest element
of the Weyl group $W_G:=N_G(\g a)/Z_G(\g a)$ of $G$.
Since $\Gamma$ is Zariski dense it contains an element 
$g_0$ in the Zariski open set $Uw_0P$. 
After replacing $\Gamma$
by a suitable conjugate $u_0\Gamma u_0^{-1}$ with $u_0$ in $U$, 
we can assume, without loss of generality,  that the element $g_0$ is in $w_0P$ so that 
the group 
$g_0\Delta g_0^{-1}$ is a lattice in $w_0 U w_0^{-1}$.
Note that, since $U$ is non-reflexive,
the conjugate $w_0Uw_0^{-1}$ intersects $P$ non-trivially.
According to Lemma \ref{lempglgug}.a 
the group $P_0:=\{ g\in P\mid \det_\g u{\rm Ad}g=1\}$
is $\Ga$-proper. Therefore, since 
the group $w_0 U w_0^{-1}$ is $\Ga$-compact, by Lemma \ref{lemgcogpr}, the group 
$$
V':= P_0\cap  w_0 U w_0^{-1}
\;\; \;
\mbox{\rm is $\Ga$-compact.}
$$
This non-trivial group $V'$ normalizes $U$ and the group 
$V:=V'U$ is also a unipotent group which is $\Ga$-compact.

We now check that $V$ is a horospherical subgroup. 
We compute its Lie algebra $\g v$.
The transformation $\iota:=-w_0$ induces a bijection of 
the set $\Pi$ of simple roots of $G$. 
One has the equality $\g v=\oplus_\al \,\g g_\al $ 
where the sum is over the roots $\al\in \Sigma$ such that
$$
\mbox{\rm $n_\theta (\alpha )>0$
\;\; or\;\; ( $n_\th(\al)=0$\; and $n_{\iota(\th)\cap \th^c}(\al)<0$ )}\, . 
$$
We note that the set $\th^c=\Pi\smallsetminus \th$ is a set of simple roots for the group $L$. We denote by $w'_0$ 
the longest element of the Weyl group 
$W_L:=N_L(\g a)/Z_L(\g a)$ of $L$. 
The transformation $j:=-w'_0$ induces a bijection of 
the set $\th^c$ of simple roots of $L$. 
Therefore, one has the equality ${\rm Ad}w'_0(\g v)=\oplus_\al \,\g g_\al $ 
where the sum is over the roots $\al\in \Sigma$ such that
$$
\mbox{\rm $n_\theta (\alpha )>0$
\;\; or\;\; ( $n_\th(\al)=0$\; and $n_{j(\iota(\th)\cap \th^c)}(\al)>0$ )} \, .
$$
This tells us that this Lie algebra is a standard horospherical
Lie algebra 
$$
{\rm Ad}w'_0(\g v)=\g u_{\th'}
$$
with $\th'=\th\cup j(\iota(\th)\cap \th^c)$.

By Lemma \ref{lemirrhor}.$i)$, the lattice 
$\Delta_V:=\Gamma\cap V$ of $V$ is irreducible.
\end{proof}

%52
\subsection{When the center of $\g u$ is a root space}
\label{secroospa}
\bq
We now explain how to reduce the case of a
reflexive horospherical subalgebra $\g u$ which is not Heisenberg 
but whose center is a root space (see Definition \ref{defroospa})
to the case of 
a reflexive horospherical subalgebra $\g u'$ whose center is not a root space.
\eq

\begin{Prop}
\label{proroospa}
Let $G$ be a semisimple real algebraic Lie group
of real rank at least two,
$U$ be a non-trivial horospherical subgroup, 
and $\Gamma$ be a Zariski dense discrete subgroup 
of $G$ that contains an irreducible lattice $\Delta$ of $U$.

Assume $\g u$ is reflexive, not Heisenberg, and the center of $\g u$ is a root space.

Then there exists a smaller reflexive horospherical subgroup 
$U'\subsetneq U$ of $G$ such that $\Gamma$
also contains an irreducible lattice $\Delta'$ of $U'$
and the center of $\g u'$ is not a root space.
\end{Prop}

\begin{proof}[Proof of Proposition \ref{proroospa}] 
We use freely the notation of \S \ref{secroosys}
and \S \ref{secgralie} and write $U=U_\th$ for a subset $\th$ 
of the set $\Pi$ of simple roots.
Since $\Delta$ is irreducible and since the center $\g z$ of $\g u$ is a root space, 
the Lie algebra $\g g$ is simple.
The group $U$ is $s$-step nilpotent, where $s=n_\th(\widetilde{\al})\geq 3$. 
The last group $C^sU$ of the descending central sequence
is the center of $U$.
Let 
$$
\th_0:=\{\al_i\in \Pi\mid \widetilde{\al}-\al_i
\;\;\mbox{\rm is a root}\;\}\, .
$$
Since the center of $\g u$ is equal to $\g g_{\widetilde{\al}}$.
One has the inclusion $\th_0\subset \th$.
Let $U'$ be the centralizer of $C^{s-1}U$ in $U$.
Since the Lie algebra of $C^{s-1}U$ is 
$\g g_{s-1}\oplus \g g_s$,
the Lie algebra of $U'$ is
$
\g u'
\;=\;
\textstyle
\oplus_{\al}
\,\g g_\al
$ 
where the  sum is over  
the roots $\al$ such that 
$$
\mbox{\rm $n_\theta (\alpha )\geq 2$
\;\; or\;\; ( $n_\th(\al)=1$\; and $n_{\th_0}(\al)=0$ )}\, . 
$$
This is also the set of roots $\al$ such that 
$n_{\th\smallsetminus\th_0}(\al)\geq 1$.
Therefore one has the equality $U'=U_{\theta\smallsetminus\theta_0}$.
By Corollary \ref{cornillat}, 
the group $\Delta\cap U'$ is a lattice in $U'$. 
This lattice is automatically irreducible since $\g g$ is simple.
\end{proof}

%53
\subsection{When the center of $\mathfrak u$ is not a root space}
\label{secnonroo}

\bq
The following proposition is the last 
of our three reduction steps for Theorem \ref{thmmar}.
It deals with a reflexive non-commutative horospherical
subgroup whose center is not a root space.  
\eq

\begin{Prop}
\label{prononroo}
Let $G$ be a semisimple real algebraic Lie group
of real rank at least two,
$U$ be a non-trivial horospherical subgroup, 
and $\Gamma$ be a Zariski dense discrete subgroup 
of $G$ that contains an irreducible lattice $\Delta$ of $U$.

Assume  $\g u$ is reflexive, non-commutative, and its center 
is not a root space.

Then the group  $\Gamma$ is an irreducible arithmetic subgroup of $G$.
\end{Prop}

Note that in the following proof we  
use our main theorem \ref{thmmar} 
for a smaller dimensional group $G'$.
This is allowed by induction. 

\begin{proof}
Since $U$ is reflexive, by Lemma \ref{lemgamuum},  
there exists an opposite horospherical subgroup $U^-$
such that the group
$\Delta^-:=\Gamma\cap U^-$ is a lattice in $U^-$.
These groups are $s$-step nilpotent.
Without loss of generality we can assume that the group $\Gamma$
is generated by $\Delta$ and $\Delta^-$. 
By \S \ref{secgralie}, there exists a graduation
\begin{equation*}
\g g=\g g_{-s}\oplus\ldots \oplus \g g_0
\oplus\ldots \oplus \g g_s
\;\;\;\mbox{\rm such that }
\end{equation*}
\begin{equation*}
\g u^{-}=\g g_{-s}
\oplus\ldots \oplus \g g_{-1}
\;\; ,\;\; \g l=\g g_0
\;\;\;{\rm and}\;\;\;
\g u=\g g_1
\oplus\ldots \oplus \g g_s\, .
\end{equation*}
Let $U':=Z$ and ${U'}^-:=Z^-$
be the centers of $U$ and  $U^-$.
The Lie algebras of $U'$ and ${U'}^-$ are $\g u'=\g z=\g g_s$
and ${\g u'}^-=\g z^-=\g g_{-s}$. 
By Corollary \ref{cornillat}.c, the  groups $\Delta':=\Gamma\cap U'$ 
and ${\Delta'}^{-}:=\Gamma\cap {U'}^-$
are lattices in $U'$ and ${U'}^-$.

Let $\Ga'$ be the discrete subgroup of $G$ generated by
$\Delta'$ and ${\Delta'}^-$.
The Zariski closure of $\Ga'$ is the group $G'$
generated by $U'$ and ${U'}^-$. 
This group $G'$ is a semisimple real algebraic Lie group
whose Lie algebra is
\begin{equation*}
\g g':=\g g_{-s}\oplus \g l'
\oplus \g g_s
\;\;\;{\rm with}\;\;\;
\g l':=[\g g_{-s},\g g_s]\subset \g g_0\, .
\end{equation*}
Since $\g u'=\g z$ is not a root space, the real rank of $G'$
is at least two.
The groups $U'$ and $U'^{-}$ are opposite horospherical 
subgroups. They are commutative.
Since $\Delta $ is irreducible, 
by Lemma \ref{lemirruua}, for every simple ideal $\g g_a$ of $\g g$, the intersection 
$\g u'_a:=\g u'\cap \g g_a$
is non-zero and is the center of $\g u_a:=\g u\cap \g g_a$. 
According to Lemma \ref{lemnonroo}, 
the intersection $\g g'_a:=\g g'\cap\g g_a$  is also a simple ideal of $\g g'$.
Therefore, since the group $\Delta$ is an irreducible lattice 
of $U$, the group $\Delta'$ is an irreducible lattice of $U'$.

Using an induction argument, we can apply  
our main theorem \ref{thmmar} 
to the smaller dimensional group ${\rm Ad}G'$:
there exists a $\m Q$-form $\g g'_\m Q$ of 
$\g g'$ such that $\Gamma'$ is commensurable with 
the stabilizer $G'_\m Z$ of $\g g'_\m Z$.
Let $P'$ and ${P'}^-$ be the normalizers of $U'$ and ${U'}^-$
in $G'$ and $L':=P'\cap{P'}^-$. 
Since $L$ normalizes $U'$ and ${U'}^-$, this group $L'$ is  a normal subgroup of $L$
whose real rank  is equal to the real rank of $G'$ and hence is at least two.
Moreover the group ${\rm Ad}L'$ is a subgroup of ${\rm Ad}G'$ which is defined over $\m Q$. 
Let $L'_0$ be the intersection of the kernels of all the characters of ${\rm Ad}L'$
that are defined over $\m Q$.

We claim that {\it the group $L'_0$ is non-compact.}\\
- In the case  $G $ is not simple, since the group $\Delta'$ is an irreducible lattice, 
the characters  $\ell\mapsto {\rm det}_{\g u'_a}({\rm Ad}\ell)$ of ${\rm Ad}L'$ are not defined over $\m Q$. 
Hence $L'_0$ is non-compact.\\ 
- In the case  $G$ is simple, the group $G'$ is simple too and since $U$ is abelian,
the split center of $L$ is one-dimensional 
and $L'_0$ is a semisimple group of real rank at least one.
Hence $L'_0$ is non-compact.

Now, according to Borel and Harish-Chandra theorem, the 
discrete subgroup 
$L'_{0,\m Z}:=L_0\cap G'_\m Z$ is a lattice of $L'_0$ hence is an infinite group.
This group stabilizes $\Delta$ and $\Delta^-$.
Therefore by Corollary \ref{corextfor} the group $\Ga$
is an irreducible arithmetic lattice of $G$.
\end{proof}

\begin{Rem}
Note that in the previous proof, the horospherical subgroup 
$U'$ of $G'$ is not always reflexive. 
For instance when $G={\rm SO}(d, \m C)$ and $P$ is the stabilizer of an isotropic $p$-plane in $\m C^d$
with $1<p<n/2$, then the group 
$G'$ is $G'={\rm SL}(p, \m C)$ and $P':=L'U'$ is the stabilizer of a line in $\m C^p$.
\end{Rem}

%54
\subsection{$S$-arithmetic setting}
\label{secsarset}

Our main result can be extended to product 
of simple groups over local fields.
We just quote the statement 
complementing our Theorem \ref{thmmar}.

\begin{Prop}
\label{promar} 
let $S$ be a finite set of 
valuation of $\m Q$ including the archi\-medean valuation $\infty$
and at least one finite valuation.
For $p$ in $S$, let 
$G_p$ be the group of $\m Q_{p}$-points  of
a  connected semisimple algebraic $\m Q_{p}$-group, 
and  $U_p$ be a horospherical subgroup of $G_p$
intersecting non-trivially all the simple factors of $G_p$.
Let $G=\prod_{p\in S}G_p$, $U:=\prod_{p\in S}U_p$.
Let $\Ga$ be a discrete subgroup of $G$ whose image in every  $G_p$ is Zariski dense. Assume that $\Gamma$ 
contains a  lattice $\Delta$ of $U$ such that $\Delta\cap U_\infty$
is irreducible in $G_\infty$. 
Then $\Ga$ is a  lattice in $G$. 
\end{Prop}

The proof is similar and ensures that $\Ga$ is $S$-arithmetic.
Indeed the main tools we used for real Lie groups 
have their counterparts on 
products of Lie groups over $\m Q_p$: the Raghunathan-Venkataramana
theorem, Margulis' construction of $\m Q$-forms,
the recurrence of unipotent flow of Dani-Margulis and 
the Howe-Moore decay of matrix coefficients.

Partial results in this direction were obtained 
by Hee Oh  in \cite[Theorem 4.3]{Oh00} when $G_\infty$ 
is a higher rank absolutely simple Lie group and 
by Benoist and Oh in \cite[Theorem 1.1]{BenoistOh10a} 
when all the groups $G_p$ are products of ${\rm SL}(2,\m Q_p)$.

%6
\section{Horospherical Lie subalgebras}
\label{secapphor}
We give the list of the horospherical Lie subalgebras that occur in Chapter \ref{secorbhor} and \ref{secarigam}, 
and we deduce from it a proof 
of Proposition \ref{profphf}.a and Lemma \ref{lemfphf}
which are the algebraic parts of the proof of our
closedness result in Chapter \ref{secorbhor}.

%61
\subsection{Reflexive commutative horospherical subalgebras}
\label{secrefcom}

\bq
The first class of horospherical subalgebras which play an important role
are the reflexive commutative ones.
\eq

Let $\g g$ be a semisimple real Lie algebra
and 
$\g u$ be a horospherical subalgebra. 
We will use freely  the notation of  \S \ref{secroosys} and \S \ref{secgralie}. One 
can choose a split Cartan subspace $\g a$ of $\g g$, a set of simple restricted roots $\Pi$ and a subset $\theta\subset \Pi$
such that 
$\g u=\g u_\th$.
This Lie algebra  $\g u_\th$ is commutative if and only if it  is the component $\g u_\th=\g g_1$ of a grading
\begin{equation*}
\label{eqng1gog1}
\g g=\g g_{- 1}  \oplus \g g_0
  \oplus \g g_1\, ,
\end{equation*}

One can easily  reduce the  classification 
of the reflexive horospherical subalgebras $\g u$ of $\g g$
to the case where $\g g$ is a complex simple Lie algebras.
Indeed, the horospherical subalgebra $\g u$
of $\g g$ is reflexive commutative if and only if 
its complexification $\g u_\m C$ is reflexive commutative in $\g g_\m C$. 
This horospherical Lie algebra is then a direct sum of  
reflexive commutative horospherical Lie algebras of the simple ideals of $\g g_\m C$.

We assume now that the Lie algebra  $\g g$ is simple.
The commutative horospherical subalgebras $\g u_\th$ 
are exactly those for which $\th$ contains  only one simple root $\th=\{\al\}$ and
such that $n_\th(\widetilde\al)=1$, where $\widetilde\al$ is the largest root of $\Sigma$.
This subalgebra is reflexive if and only if 
$\al=-w_0\al$. It is therefore very easy to give the list  of the reflexive 
commutative horospherical subgroups: 
one expresses the largest root as a sum of simple roots $\widetilde\al=\sum_in_i\al_i$ 
and pick out those simple roots for which $n_i=1$ and $-w_0\al_i=\al_i$.

%\end{document}

%T1
%\input{table1.tex}

\begin{table}[ht!]
\centerline{$\begin{array}{|cc|c|c|}
\hline
\raisebox{3ex}{ }\raisebox{-1.5ex}{ }
& { \g g}&[\g l,\g l] & F(X)\\
\multicolumn{2}{|c|}{\rm Simple\; roots}  &
\g u\simeq \g u^- & G(X,Y)\\
\hline
(1)& 
\raisebox{3ex}{ }\raisebox{-1.5ex}{ }
\g s\g l (\m C^{2n})\;\;   n\geq 3 &
\g s\g l (\m C^{n})\!\oplus\! \g s\g l (\m C^{n})&
({\rm det}X)^{2n}\\
%A6
\multicolumn{2}{|c|}{\mbox{
\begin{tikzpicture}[scale=0.8]
\draw 
(0,0) --++ (1,0)
(1,0) --++ (1,0)
(2,0) --++ (1,0)
(3,0) --++ (1,0);
\draw[fill=white]
(0,0) circle [radius=.1] 
(1,0) circle [radius=.1] 
(3,0) circle [radius=.1] 
(4,0) circle [radius=.1];
\draw[fill=white]
(2,0) circle [radius=.05]
(2,0) circle [radius=.1];
\end{tikzpicture}
}}
& {\rm End }(\m C^n)&({\rm det}(1+XY))^{2n}\\
\hline
(2)&
\raisebox{3ex}{ }\raisebox{-1.5ex}{ }
\g s\g o (\m C^{n+2})\;\;
 n\geq 3&
\g s\g o (\m C^{n})&
q(X)^{n}\\
%B5 et D6
\multicolumn{2}{|c|}{\mbox{
\begin{tikzpicture}[scale=0.8]
\draw 
(0,0) --++ (1,0)
(1,0) --++ (1,0)
(2,0) --++ (1,0)
(3,-.04) --++ (1,0)
(3,+.04) --++ (1,0); 
\draw
(3.5,0) --++ (130:.2)
(3.5,0) --++ (-130:.2);
\draw[fill=white]
(1,0) circle [radius=.1] 
(2,0) circle [radius=.1] 
(3,0) circle [radius=.1] 
(4,0) circle [radius=.1];
\draw[fill=white]
(0,0) circle [radius=.05]
(0,0) circle [radius=.1];
\draw 
(0,-0.8) --++ (1,0)
(1,-0.8) --++ (1,0)
(2,-0.8) --++ (1,0)
(3,-0.8) --++ (0.9,0.3)
(3,-0.8) --++ (0.9,-0.3); 
\draw[fill=white]
(1,-0.8) circle [radius=.1] 
(2,-0.8) circle [radius=.1] 
(3,-0.8) circle [radius=.1] 
(3.9,-0.5) circle [radius=.1]  
(3.9,-1.1) circle [radius=.1];
\draw[fill=white]
(0,-0.8) circle [radius=.05]
(0,-0.8) circle [radius=.1];
\end{tikzpicture}
}}
& \raisebox{2ex}{$\m C^n$}&
\raisebox{2ex}{$(1\!-\!b(X,Y)\!+\!\tfrac14 q(X)q(Y))^n$}\\
\hline
(3)&
\raisebox{3ex}{ }\raisebox{-1.5ex}{ }
\g s\g p (\m C^{2n})\;\;
 n\geq 3&
\g s\g l (\m C^{n})&
({\rm det}X)^{n+1}\\
%C5
\multicolumn{2}{|c|}{\mbox{
\begin{tikzpicture}[scale=0.8]
\draw 
(0,0) --++ (1,0)
(1,0) --++ (1,0)
(2,0) --++ (1,0)
(3,-.04) --++ (1,0)
(3,+.04) --++ (1,0); 
\draw
(3.5,0) --++ (50:.2)
(3.5,0) --++ (-50:.2);
\draw[fill=white]
(0,0) circle [radius=.1] 
(1,0) circle [radius=.1] 
(2,0) circle [radius=.1] 
(3,0) circle [radius=.1];
\draw[fill=white]
(4,0) circle [radius=.05]
(4,0) circle [radius=.1];
\end{tikzpicture}
}}
& S^2(\m C^n)&({\rm det}(1+XY))^{n+1}\\
\hline
(4)&
\raisebox{3ex}{ }\raisebox{-1.5ex}{ }
\g s\g o (\m C^{2n})\;
 n\,{\rm pair}\!\geq\! 4&
\g s\g l (\m C^{n})&
({\rm det}X)^{n-1}\\
%D6
\multicolumn{2}{|c|}{\mbox{
\begin{tikzpicture}[scale=0.8]
\draw 
(0,0) --++ (1,0)
(1,0) --++ (1,0)
(2,0) --++ (1,0)
(3,0) --++ (0.9,0.3)
(3,0) --++ (0.9,-0.3); 
\draw[fill=white]
(1,0) circle [radius=.1] 
(2,0) circle [radius=.1] 
(3,0) circle [radius=.1] 
(0,0) circle [radius=.1] 
(3.9,-0.3) circle [radius=.1]; 
\draw[fill=white]
(3.9,0.3) circle [radius=.05]
(3.9,0.3) circle [radius=.1];
\end{tikzpicture}
}}
& \Lambda^2(\m C^n)&({\rm det}(1+XY))^{n-1}\\
\hline
(5)&
\raisebox{3ex}{ }\raisebox{-1.5ex}{ }
\g e_7\hspace{2em}\mbox{ }&
\g e_6&
I_3(X)^{18}
\\
%E7
\multicolumn{2}{|c|}{\mbox{
\begin{tikzpicture}[scale=0.8]
\draw 
(0,0) --++ (1,0)
(1,0) --++ (1,0)
(2,0) --++ (1,0)
(3,0) --++ (1,0)
(4,0) --++ (1,0)
(2,0) --++ (0,-0.6);
\draw[fill=white]
(0,0) circle [radius=.1] 
(1,0) circle [radius=.1] 
(2,0) circle [radius=.1] 
(2,-0.6) circle [radius=.1] 
(3,0) circle [radius=.1] 
(4,0) circle [radius=.1];
\draw[fill=white]
(5,0) circle [radius=.05]
(5,0) circle [radius=.1];
\end{tikzpicture}
}}
& \m C^{27}&G(X,Y)\\
\hline
\end{array}$}
\caption{Reflexive commutative horospherical $\g u$ in complex simple $\g g$}
\label{tabhorcom}
\end{table}

\begin{Rem} We want to point out that,
when $\g g$ is a complex  Lie algebra,  
the pair $(\g g, \g g_0)$ is the complexification of a hermitian symmetric pair 
and the map $(\g g,\g u)\ra (\g g, \g g_0)$ induces a bijection between\\
$\{$complex abelian horospherical   algebras$\}$
and $\{$hermitian symmetric spaces$\}$.
The reflexive commutative complex horospherical 
subalgebras correspond to the 
hermitian symmetric spaces of {\it tube type}
(see \cite[p.528]{Helgason78} or \cite[p.97]{FarautKoranyi94}).
\end{Rem}
 
We give in Table \ref{tabhorcom} this list of reflexive commutative horospherical Lie algebras $\g u$ in a complex simple Lie algebra $\g g$.
Choosing real forms for these pairs $(\g g,\g u)$ gives 
all  the reflexive commutative horospherical Lie algebras $\g u$ in a real absolutely simple Lie algebra $\g g$.
We give also the polynomials $F$ and $G$ 
introduced in 
\eqref{eqnpolfg}. In this table,
$q$ is a non-degenerate quadratic form 
on $\m C^{2n}$, $b$ 
the associated bilinear form
and $I_3(X)$  a cubic form on $\m C^{27}$.

%62
\subsection{Heisenberg horospherical subalgebras}
\label{secheihor}
\bq
The second class of horospherical Lie subalgebras which play an important role
are the Heisenberg horospherical subalgebras.
\eq

Let $\g g$ be a semisimple real algebraic Lie group
and 
$\g u$ be a horospherical subalgebra. 
Using again the notation of \S \ref{secroosys}
and \S \ref{secgralie}, one 
can choose a Cartan subspace $\g a$ of $\g g$, 
a set of simple restricted root $\Pi$ and a subset $\theta\subset \Pi$
such that 
$\g u=\g u_\th$.
By Definition \ref{defheihor}, 
a horospherical Lie subalgebra  $\g u_\th$ of $\g g$   is Heisenberg if $\g g$ is simple and if  $\g u_\th$ is 
the sum $\g u_\th=\g g_1\oplus \g g_2$ of a grading
\begin{equation*}
\label{eqng2gog2}
\g g=\g g_{- 2}  \oplus \g g_{- 1}  \oplus \g g_0
  \oplus \g g_1\oplus \g g_2\, ,
\end{equation*}
where $\g g_1\neq 0$ and $\g g_2=\g g_{\widetilde{\al}}$.

The existence of a Heisenberg horospherical 
subalgebra   implies 
that $\g g$ is   not isomorphic to 
$\g s\g o(d,1)$.
The next lemma tells us that the converse is true.

%T2
%\input{table2.tex}
\begin{table}[ht!]
\centerline{$\begin{array}{|cc|c|}
\hline
\raisebox{3ex}{ }\raisebox{-1.5ex}{ }
& { \g g}&[\g l,\g l] \\
\multicolumn{2}{|c|}{\rm Simple\; roots}  &
\g u\simeq \g u^-   \\
\hline
(1)&
\raisebox{3ex}{ }\raisebox{-1.5ex}{ }
\g s\g l (\m C^{n+2})\;\;   n\geq 1 &
\g s\g l (\m C^{n})\\
%A5
\multicolumn{2}{|c|}{\mbox{
\begin{tikzpicture}[scale=0.8]
\draw 
(0,0) --++ (1,0)
(1,0) --++ (1,0)
(2,0) --++ (1,0)
(3,0) --++ (1,0);
\draw[fill=white]
(1,0) circle [radius=.1] 
(2,0) circle [radius=.1] 
(3,0) circle [radius=.1];
\draw[fill=white]
(0,0) circle [radius=.05] 
(0,0) circle [radius=.1] 
(4,0) circle [radius=.05]
(4,0) circle [radius=.1];
\end{tikzpicture}
}}
&{\m C^n}^*\oplus {\m C^n}\oplus \m C\\
\hline
(2)&
\raisebox{3ex}{ }\raisebox{-2ex}{ }
\g s\g o (\m C^{n+4})\;\;
 n\geq 2&
\g s\g l(\m C^2)\oplus\g s\g o (\m C^{n})\\
%B5 et D6
\multicolumn{2}{|c|}{\mbox{
\begin{tikzpicture}[scale=0.8]
\draw 
(0,0) --++ (1,0)
(1,0) --++ (1,0)
(2,0) --++ (1,0)
(3,-.04) --++ (1,0)
(3,+.04) --++ (1,0); 
\draw
(3.5,0) --++ (130:.2)
(3.5,0) --++ (-130:.2);
\draw[fill=white]
(0,0) circle [radius=.1] 
(2,0) circle [radius=.1] 
(3,0) circle [radius=.1] 
(4,0) circle [radius=.1];
\draw[fill=white]
(1,0) circle [radius=.05]
(1,0) circle [radius=.1];
\draw 
(0,-0.8) --++ (1,0)
(1,-0.8) --++ (1,0)
(2,-0.8) --++ (1,0)
(3,-0.8) --++ (0.9,0.3)
(3,-0.8) --++ (0.9,-0.3); 
\draw[fill=white]
(0,-0.8) circle [radius=.1] 
(2,-0.8) circle [radius=.1] 
(3,-0.8) circle [radius=.1] 
(3.9,-0.5) circle [radius=.1]  
(3.9,-1.1) circle [radius=.1];
\draw[fill=white]
(1,-0.8) circle [radius=.05]
(1,-0.8) circle [radius=.1];
\end{tikzpicture}
}}
& \raisebox{2ex}{$\m C^2\otimes\m C^n\oplus\m C$}\\
\hline
(3)&
\raisebox{3ex}{ }\raisebox{-1.5ex}{ }
\g s\g p (\m C^{2n+2})\;\;
 n\geq 1&
\g s\g p (\m C^{2n})\\
%C5
\multicolumn{2}{|c|}{\mbox{
\begin{tikzpicture}[scale=0.8]
\draw 
(0,0) --++ (1,0)
(1,0) --++ (1,0)
(2,0) --++ (1,0)
(3,-.04) --++ (1,0)
(3,+.04) --++ (1,0); 
\draw
(3.5,0) --++ (50:.2)
(3.5,0) --++ (-50:.2);
\draw[fill=white]
(1,0) circle [radius=.1] 
(2,0) circle [radius=.1] 
(3,0) circle [radius=.1] 
(4,0) circle [radius=.1];
\draw[fill=white]
(0,0) circle [radius=.05]
(0,0) circle [radius=.1];
\end{tikzpicture}
}}
& \m C^{2n}\oplus \m C\\
\hline
(4)&
\raisebox{3ex}{ }\raisebox{-1.5ex}{ }
\g e_6&
\g s\g l(\m C^6)\\
%E6
\multicolumn{2}{|c|}{\mbox{
\begin{tikzpicture}[scale=0.8]
\draw 
(0,0) --++ (1,0)
(1,0) --++ (1,0)
(2,0) --++ (1,0)
(3,0) --++ (1,0)
(2,0) --++ (0,-0.6);
\draw[fill=white]
(0,0) circle [radius=.1] 
(1,0) circle [radius=.1] 
(2,0) circle [radius=.1] 
(3,0) circle [radius=.1] 
(4,0) circle [radius=.1];
\draw[fill=white]
(2,-0.6) circle [radius=.05]
(2,-0.6) circle [radius=.1]; 
\end{tikzpicture}
}}
& \Lambda^3( \m C^{6})\oplus \m C\\
\hline
(5)&
\raisebox{3ex}{ }\raisebox{-1.5ex}{ }
\g e_7&
\g s\g o(\m C^{12})\\
%E7
\multicolumn{2}{|c|}{\mbox{
\begin{tikzpicture}[scale=0.8]
\draw 
(0,0) --++ (1,0)
(1,0) --++ (1,0)
(2,0) --++ (1,0)
(3,0) --++ (1,0)
(4,0) --++ (1,0)
(2,0) --++ (0,-0.6);
\draw[fill=white]
(1,0) circle [radius=.1] 
(2,0) circle [radius=.1] 
(2,-0.6) circle [radius=.1] 
(3,0) circle [radius=.1] 
(4,0) circle [radius=.1] 
(5,0) circle [radius=.1];
\draw[fill=white]
(0,0) circle [radius=.05]
(0,0) circle [radius=.1];
\end{tikzpicture}
}}
& \m C^{32}\oplus \m C\\
\hline
(6)&
\raisebox{3ex}{ }\raisebox{-1.5ex}{ }
\g e_8&
\g e_7\\
%E8
\multicolumn{2}{|c|}{\mbox{
\begin{tikzpicture}[scale=0.8]
\draw 
(0,0) --++ (1,0)
(1,0) --++ (1,0)
(2,0) --++ (1,0)
(3,0) --++ (1,0)
(4,0) --++ (1,0)
(5,0) --++ (1,0)
(2,0) --++ (0,-0.6);
\draw[fill=white]
(0,0) circle [radius=.1] 
(1,0) circle [radius=.1] 
(2,0) circle [radius=.1] 
(2,-0.6) circle [radius=.1] 
(3,0) circle [radius=.1] 
(4,0) circle [radius=.1] 
(5,0) circle [radius=.1];
\draw[fill=white]
(6,0) circle [radius=.05]
(6,0) circle [radius=.1];
\end{tikzpicture}
}}
& \m C^{56}\oplus\m C\\
\hline
(7)&
\raisebox{3ex}{ }\raisebox{-1.5ex}{ }
\g f_4 &
\g s\g p (\m C^{6})\\
%F4
&\mbox{
\begin{tikzpicture}[scale=0.8]
\draw 
(0,0) --++ (1,0)
(1,-.04) --++ (1,0)
(1,+.04) --++ (1,0)
(2,0) --++ (1,0); 
\draw
(1.5,0) --++ (130:.2)
(1.5,0) --++ (-130:.2);
\draw[fill=white]
(1,0) circle [radius=.1] 
(2,0) circle [radius=.1] 
(3,0) circle [radius=.1];
\draw[fill=white]
(0,0) circle [radius=.05]
(0,0) circle [radius=.1];
\end{tikzpicture}
}
& \Lambda^3_0(\m C^6)\oplus\m C\\
\hline
(8)&
\raisebox{3ex}{ }\raisebox{-1.5ex}{ }
\g g_2 &
\g s\g l (\m C^{2})\\
%G2
&\mbox{
\begin{tikzpicture}[scale=0.8]
\draw 
(0,-0.08) --++ (1,0)
(0,0) --++ (1,0)
(0,0.08) --++ (1,0); 
\draw
(0.5,0) --++ (50:.2)
(0.5,0) --++ (-50:.2);
\draw[fill=white]
(0,0) circle [radius=.1];
\draw[fill=white]
(1,0) circle [radius=.05]
(1,0) circle [radius=.1];
\end{tikzpicture}
}
& S^3(\m C^2)\oplus \m C\\
\hline
\end{array}$}
\caption{  Complex   
Heisenberg horospherical $\g u$}
\label{tabhorhei}
\end{table}

\begin{Lem}
\label{lemheihor} 
Let $\g g$ be a simple real Lie algebra not
isomorphic to $\g s\g o(d,1)$.
Then $\g g$ contains a Heisenberg horospherical subalgebra $\g u$.
This Lie subalgebra is unique up to inner conjugation.
In particular it is reflexive.   
\end{Lem}

The Lie algebras $\g g=\g s\g o(d,1)$ play a special role since 
they are the only simple Lie algebras 
for which the largest root $\widetilde\al$ is also
a simple root.

\begin{proof}
If the horospherical Lie subalgebra
$\g u_\th$ is Heisenberg, one must have
\begin{equation}
\label{eqnthhei}
\th=\{\al_i\in \Pi\mid \widetilde\al -\al_i\in \Sigma\}.
\end{equation}
This proves that $\g g$ contains at most one Heisenberg
horospherical subalgebra up to inner conjugacy. 

Let us check that the horospherical subalgebra $\g u_\th$ with $\th$ given by 
\eqref{eqnthhei} is Heisenberg. Indeed, one has
$  (\widetilde\al,\al_i)=\|\widetilde\al\|^2/2$ for $\al_i\in \th$ and 
$  (\widetilde\al,\al_i)=0$ for $\al_i\in \th^c$. 
One expresses the largest root as a sum of simple roots $\widetilde\al=\sum_in_i\al_i$ 
and one computes
$
\sum_{\al_i\in \th} n_i=
2\sum_i n_i \frac{(\widetilde\al,\al_i)}{\|\widetilde\al\|^2}=2.
$
\end{proof}

In Table \ref{tabhorhei}, we give the list of the 
pairs $(\g g,\g u)$ where $\g g$ is a complex simple Lie algebra and $\g u$ is a  
Heisenberg horospherical  Lie subalgebra,
the grey points 
corresponding to the subset  $\th$ of $\Pi$. 

When $\g u$ is a Heisenberg horospherical subalgebra in an absolutely simple real Lie algebra $\g g$ 
with $\dim \g z=1$,   the pair $(\g g_\m C,\g u_\m C)$
is still Heisenberg and is in Table  \ref{tabhorhei}.

%T3
%\input{table3.tex}

\begin{table}[ht!]
\centerline{$\begin{array}{|cc|c|c|}
\hline
\raisebox{3ex}{ }\raisebox{-1.5ex}{ }
&{\g g}\mbox{ }& {\rm Simple\;
restricted\; roots} &[\g l,\g l] \\
&{ \g g_\m C}& \mbox{\rm Satake diagram}  &
\g u\simeq \g u^- \\
\hline
(9)&
\raisebox{3ex}{ }\raisebox{-1.5ex}{ }
\g s\g l (\m H^{n+2})&
%A4
\mbox{
\begin{tikzpicture}[scale=0.8]
\draw 
(0,0) --++ (1,0)
(1,0) --++ (1,0)
(2,0) --++ (1,0);
\draw[fill=white]
(1,0) circle [radius=.1] 
(2,0) circle [radius=.1]; 
\draw[fill=white]
(0,0) circle [radius=.05] 
(0,0) circle [radius=.1] 
(3,0) circle [radius=.05] 
(3,0) circle [radius=.1];
\end{tikzpicture}
}&
\g s\g l (\m H^{n})\oplus\g s\g l (\m H)^2\\
n\geq 1 &
\g s\g l (\m C^{2n+4})&
%A9
\mbox{
\begin{tikzpicture}[scale=0.6]
\draw 
(0,0) --++ (1,0)
(1,0) --++ (1,0)
(2,0) --++ (1,0)
(3,0) --++ (1,0)
(4,0) --++ (1,0)
(5,0) --++ (1,0)
(6,0) --++ (1,0)
(7,0) --++ (1,0);
\draw[fill=white]
(3,0) circle [radius=.12] 
(5,0) circle [radius=.12];
\draw[fill=white]
(1,0) circle [radius=.06] 
(1,0) circle [radius=.12] 
(7,0) circle [radius=.06]
(7,0) circle [radius=.12];
\draw[fill=black]
(0,0) circle [radius=.12] 
(2,0) circle [radius=.12] 
(4,0) circle [radius=.12] 
(6,0) circle [radius=.12] 
(8,0) circle [radius=.12];
\end{tikzpicture}
}&
{\m H^n}^*\oplus {\m H^n}\oplus \m H\\
\hline
(10)&
\raisebox{3ex}{ }\raisebox{-1.5ex}{ }
\g s\g p (\m H^{p+1,q+1}) &
%C4
\mbox{
\begin{tikzpicture}[scale=0.8]
\draw 
(0,0) --++ (1,0)
(1,0) --++ (1,0)
(2,-.04) --++ (1,0)
(2,+.04) --++ (1,0); 
\draw
(2.5,0) --++ (130:.2)
(2.5,0) --++ (-130:.2);
\draw[fill=white]
(1,0) circle [radius=.1] 
(2,0) circle [radius=.1] 
(3,0) circle [radius=.1];
\draw[fill=white]
(0,0) circle [radius=.05]
(0,0) circle [radius=.1];
\end{tikzpicture}
}&
\g s\g p (\m H^{p,q})\oplus\g s\g l (\m H)\\
\! 1\! \leq\! p\!<\! q\! &
\g s\g p (\m C^{2p+2q+4}) &
%C8
\mbox{
\begin{tikzpicture}[scale=0.6]
\draw 
(0,0) --++ (1,0)
(1,0) --++ (1,0)
(2,0) --++ (1,0)
(3,0) --++ (1,0)
(4,0) --++ (1,0)
(5,0) --++ (1,0)
(6,0) --++ (1,0)
(7,-.04) --++ (1,0)
(7,+.04) --++ (1,0); 
\draw
(7.5,0) --++ (50:.2)
(7.5,0) --++ (-50:.2);
\draw[fill=white]
(3,0) circle [radius=.12] 
(5,0) circle [radius=.12] 
(7,0) circle [radius=.12];
\draw[fill=black]
(0,0) circle [radius=.12] 
(2,0) circle [radius=.12] 
(4,0) circle [radius=.12] 
(6,0) circle [radius=.12] 
(8,0) circle [radius=.12];
\draw[fill=white]
(1,0) circle [radius=.06]
(1,0) circle [radius=.12];
\end{tikzpicture}
}&
\m H^{p,q}\oplus{\rm Im}(\m H)\\
\hline
(11)&
\raisebox{3ex}{ }\raisebox{-1.5ex}{ }
\g e_{6(-5)} &
%A2
\mbox{
\begin{tikzpicture}[scale=0.8]
\draw 
(0,0) --++ (1,0);
\draw[fill=white]
(0,0) circle [radius=.05] 
(0,0) circle [radius=.1] 
(1,0) circle [radius=.05] 
(1,0) circle [radius=.1];
\end{tikzpicture}
}&
\g s\g 0 (8)\\
&\g e_6&\raisebox{3ex}{ }\raisebox{-1.5ex}{ }
%E6
\mbox{
\begin{tikzpicture}[scale=0.6]
\draw 
(0,0) --++ (1,0)
(1,0) --++ (1,0)
(2,0) --++ (1,0)
(3,0) --++ (1,0)
(2,0) --++ (0,-0.6);
\draw[fill=white]
(0,0) circle [radius=.06]
(0,0) circle [radius=.12] 
(4,0) circle [radius=.06] 
(4,0) circle [radius=.12];
\draw[fill=black]
(1,0) circle [radius=.12] 
(2,0) circle [radius=.12] 
(3,0) circle [radius=.12] 
(2,-0.6) circle [radius=.12]; 
\end{tikzpicture}
}&
\m R^8\oplus {\m R^8}\oplus \m R^8\\
\hline
\end{array}$}
\caption{  
Non-complex Heisenberg   $\g u$
with $\dim(\g z)>1$}
\label{tabhorrea}
\end{table}

In Table \ref{tabhorrea}, we give the list of the 
remaining pairs $(\g g,\g u)$, i.e. those for which $\g g$ is an absolutely simple real Lie algebra
and $\g u$ is a  
Heisenberg horospherical  Lie subalgebra 
with $\dim \g z>1$,
the grey points 
corresponding to the subset  $\th$ of $\Pi$, 
and the black points to the compact simple roots of $\g g_\m C$.

Note that, as in Tables \ref{tabhorcom} and \ref{tabhorhei}, the Lie algebra $\g l$ and 
the representation of $\g l$ in $\g u$ can easily be determined by 
a glance at the Dynkin or Satake diagram (see
\cite[Prop. 3.1]{Kac80}).
\vs

Note that the set $\th$ in \eqref{eqnthhei} is either a singleton $\th=\{\al_i\}$
with $-w_0\al_i=\al_i$ and $n_i=2$,
or a pair $\th=\{\al_i,\al_j\}$ with $-w_0\al_i=\al_j$ and $n_i=n_j=1$.
This second case occurs if and only if the root system $\Si$
is of type $A_r$, i.e. in Cases $(1)$, $(9)$
and $(11)$ of Tables \ref{tabhorhei} and 
\ref{tabhorrea}.

\begin{Exa}
We have explicitely computed in Table \ref{tabhorcom},
the polynomials $F(X)$ for reflexive commutative Lie subalgebras.
One can also explicitely compute the polynomials 
$F(X)$ for the Heisenberg Lie subalgebras
in Tables \ref{tabhorhei} and \ref{tabhorrea}. 
Here are three examples where $F$ is given up to a scalar multiple:\\   
- Case (1): One has $\g g=\g s\g l(\m R^{n+2})$,
$\g l=\g s\g l(\m R^{n})\oplus\m R^2$,\\
$\g u=\{X=(f,v,z)\in {\m R^n}^*\oplus\m R^n\oplus\m R\}$, and
$
F(X)\simeq f(v)^2-z^2\, .
$\\
- Case (2): One has $\g g=\g s\g 0(2,n+2)$,
$\g l=\g s\g l(\m R^{2})\oplus\g s\g o(\m R^n)\oplus \m R^2$,\\
$\g u=\{X=(V,z)\in \m R^2\otimes\m R^n\oplus\m R\}$, 
and
$
F(X)\simeq \det(V\!\mbox{ }^tV)-z^2\, .
$\\
- Case (3): One has $\g g=\g s\g p(\m R^{2n+4})$,
$\g l=\g s\g p(\m R^{2n})\oplus\m R^2$,\\
$\g u=\{X=(v,z)\in \m R^{2n}\oplus\m R\}$, and
$
F(X)\simeq z^2\, .
$
\end{Exa}

%63
\subsection{When $U$ is reflexive commutative}
\label{secpolcom}
\bq
The aim of this section is to prove 
Proposition \ref{profphf} for a reflexive commutative horospherical subgroup.
\eq

We first discuss a few   facts about reflexive commutative horospherical 
subalgebras.
We use the notation of \S \ref{secrefcom}: 
one has $\g l=\g g_0$ and $\g u=\g g_1$ for some graduation 
$$
\g g=\g g_{-1}\oplus\g g_0\oplus\g g_1.
$$ 
We denote by 
$h_0$ the element of the center of $\g g_0$ inducing this graduation,
i.e. such that ${\rm ad}h_0=\pm 1$ on $\g g_{\pm 1}$.
 
We set $\g u_a:=\g u\cap \g g_a$ where $\g g_a$ are the simple ideals of $\g g$.
The function 
$$
F(X)={\rm det}_{\g u}(\tfrac12({\rm ad}X)^2w_0)
$$ 
introduced   in 
\eqref{eqnpolfg} for $X\in \g g_1$ is also given by 
\begin{equation}
\label{eqnfxfaxa}
\textstyle
F(X)=\prod_a F_a(X_a)
\end{equation}
where $X=\sum_a X_a$ with $X_a\in \g u_a$ and
\begin{equation}
\label{eqnfxdetu}
F_a(X_a)={\rm det}_{\g u_a}(\tfrac12({\rm ad}X_a)^2w_0)
\end{equation}

\begin{Lem}
\label{lemfgcom}
With the notation of Proposition \ref{profphf}. 
We assume that $\g u$ is reflexive commutative. Then\\
$a)$ The polynomial $F$ is non-zero and homogeneous of degree $2d:= 2\dim\g u$.\\
$b)$ There exists $x'_0\in \g g_1$ and $y'_0\in \g g_{-1}$ 
such that $h'_0:=[x'_0,y'_0]=2\, h_0$.\\  
$c)$ When $\g g$ is absolutely simple, the Lie algebra $\g l$
is maximal in ${\g g\g l}_\m R(\g u)$.\\   
$d)$ When $\g g$ is a simple complex Lie algebra, 
the only intermediate Lie algebra
$\g h$ between $\g l$
and  ${\g g\g l}_\m R(\g u)$ is ${\g g\g l}_\m C(\g u)$.
\end{Lem}
This triple 
$(x'_0,h'_0,y'_0)$ is an $\g s\g l_2$-triple, this means that
one has the bracket relations
$[h'_0,x'_0 ]=2\, x'_0$, $[h'_0,y'_0 ]=-2\, y'_0$ and  $[x'_0,y'_0]=  h'_0$.

\begin{proof}[Proof of Lemma \ref{lemfgcom}] 
$a)$ By \eqref{eqnfxdetu} the polynomial $F$ is homogeneous of degree $2d$.
If  $F\equiv 0$, using \eqref{eqnphlvglu}, one  deduces that the polynomial
$\Phi$ is zero on the open dense set $P^-Uw_0P$, contradicting the equality $\Phi(e)=1$.

$b)$ Without loss of generality, we can assume that $\g g$ is simple.  We give a short proof 
adapted from \cite{Rubenthaler82}. 
Let $x'_0\in \g u$ such that $F(x'_0)\neq 0$.
Since $F$ is $L$-semi-invariant with character 
$\chi^2$, for  $h\in \g l$, one has,
$$
{\rm d}F(x'_0)([h,x'_0])= 2\,{\rm d}\chi(h)\, F(x'_0).
$$
The Killing form $B$ 
identifies $\g u^-$ with the dual of $\g u$
so that 
$$
y'_0:=\frac{B(h_0,h_0)}{{\rm d}\chi(h_0)}\,\frac{{\rm d}F(x'_0)}{F(x'_0)}
$$
is an element of $\g u^-$. One computes, for  $h\in \g l$,
$$
B([x'_0,y'_0],h)=B(y'_0,[h,x'_0])=2\frac{B(h_0,h_0)}{{\rm d}\chi(h_0)}\,{\rm d}\chi(h)
=2\,B(h_0,h).
$$
Therefore, one has $[x'_0,y'_0]=2h_0$.

$c)$  Since $\g g_\m C$ is simple, the representation of $L_\m C$
in $\g u_\m C$ is irreducible. 
The maximality follows from a direct analysis of each case occuring in Table \ref{tabhorcom}.
This straightforward analysis can also be seen as a special case of
Dynkin's classification of maximal 
semisimple Lie subalgebras
of classical complex simple Lie algebras in \cite{Dynkin00} . 
Indeed, since none of the examples of Table \ref{tabhorcom} belong to 
Dynkin's list of exceptions   \cite[Table 7 p.236]{OnishchikVinverg94},
we deduce that $\g l_\m C$ is a maximal subalgebra of ${\g g\g l}(\g u_\m C)$.

$d)$ Since $\g l$ and $\g u$ are complex Lie algebras, 
the complexified Lie algebra $\g l_\m C$ is the sum of two ideals
$\g l_\m C=\g l_+\oplus\g l_-$ both isomorphic to $\g l$ and 
the representation of $\g l_\m C$ in $\g u_\m C$ 
decomposes as the sum of two 
irreducible representations of 
$\g l_+$ and $\g l_-$, $\; \g u_\m C=\g u_+\oplus\g u_-$. 
Therefore 
$$
\g u_\m C\otimes \g u_{\m C}^*\; =\;
\g u_+\otimes \g u_+^*\;\oplus\;
\g u_+\otimes \g u_-^*\;\oplus\;
\g u_-\otimes \g u_+^*\;\oplus\;
\g u_-\otimes \g u_-^*
$$
Since $\g u$ is not self-dual as a representation of $\g l$, the two representations in the middle are irreducible and inequivalent.
Hence either $\g h_\m C$ contains both of them,
or $\g h_\m C$ is included in ${\g g\g l}(\g u_+)\oplus{\g g\g l}(\g u_-)$. This proves our claim since, by $c)$, the Lie subalgebra $\g l$ is   maximal  in ${\g g\g l}_\m C(\g u)$.
\end{proof}

\begin{Rem} We want to point out that the proof of $d)$ is 
valid even in Case $(2)$ of Table \ref{tabhorcom}. 
Indeed, in this case, the representation of 
$[\g l,\g l]={\g s\g o}(n, \m C)$ in $\g u=\m C^n$ is self-dual, but the representation of $\g l$ in $\g u$ is not self dual.
\end{Rem}

\begin{proof}[Proof of Proposition \ref{profphf}.a
for $\g u$ reflexive commutative ] \mbox{ }

Formula \eqref{eqnfxfaxa} implies that the group  $H$ permutes the 
factors $\g u_a$.
Hence we can assume that $\g g$ is simple. 
By Lemma \ref{lemfgcom}.a, the   function $F$ on $\g u$ is $L$-semi-invariant and non-constant. 

Assume first that $\g g$ is absolutely simple.
Since ${\rm SL}(\g u)$
has an open orbit in $\g u$, this function $F$ 
is not ${\rm GL}(\g u)$-semi-invariant.
Therefore, by   Lemma \ref{lemfgcom}.c
the Lie algebra $\g h$ of the group $H$ is equal to $\g l$.

Assume now that $\g g$ has a complex structure.
Since ${\rm SL}_\m C(\g u)$
also has an open orbit in $\g u$, this function $F$ 
is  not ${\rm GL}_\m C(\g u)$-semi-invariant.
Therefore, by Lemma \ref{lemfgcom}.d
the lie algebra $\g h$ of the group $H$ is equal to $\g l$.
\end{proof}

\begin{proof}[Proof of Lemma \ref{lemfphf}
for $\g u$ reflexive commutative ] \mbox{ }

The bilinear form $G_2$ introduced in \eqref{eqng2xy} is non-zero. Indeed, using the $\g s\g l_2$-triple 
$(x'_0,h'_0,y'_0)$ of  Lemma \ref{lemfgcom}.b,
one has
$G_2(x'_0, y'_0)={\rm tr}_\g z({\rm ad} h'_0)= 2d$.
Since this bilinear form $G_2$ is $L$-invariant and the action of $L$ on $\g u$ is irreducible, 
this bilinear form is non-degenerate.
\end{proof}

%64
\subsection{When $U$ is Heisenberg}
\label{secpolhei}
\bq
The aim of this section is to prove 
Proposition \ref{profphf} for a 
Heisenberg horospherical subgroup.
\eq

We use the notation of \S \ref{secheihor}:
$\g g$ is simple, 
$\g l=\g g_0$ and $\g u=\g u_\th=\g g_1\oplus \g g_2$ for some subset $\th\subset\Pi$ and some graduation 
\begin{equation}
\label{eqng2g0g2}
\g g=\g g_{-2}\oplus\g g_{-1}\oplus\g g_0\oplus\g g_1\oplus\g g_2.
\end{equation}
We denote by 
$h_0$ the element of the center of $\g g_0$ inducing this graduation,
i.e. such that ${\rm ad}\, h_0=j$ on $\g g_{j}$ for all $j$.

In this case, the polynomial function $F(X)$
introduced in \eqref{defpolfg} 
is not a homogeneous function anymore. 
Indeed for $X=V+Z$ with $V\in \g v:=\g g_1$ and $Z\in \g z:=\g g_2$,
it is given by
\begin{equation}
\label{eqnfvzdet}
F(V+Z)={\rm det}_\g z(
\tfrac12({\rm ad}Z)^2w_0+
\tfrac12{\rm ad}Z\,({\rm ad}V)^2w_0+
\tfrac{1}{24}({\rm ad}V)^4w_0)
\end{equation}
Its homogeneous components of degree $2d$  and $4d$, with 
$d:=\dim\g z$ are
\begin{eqnarray}
\label{eqnf2d}
F_{2d}(V+Z)&=&{\rm det}_\g z(
\tfrac12({\rm ad}Z)^2w_0)
\, ,\\
\label{eqnf4d}
F_{4d}(V+Z)&=&{\rm det}_\g z(
\tfrac{1}{24}({\rm ad}V)^4w_0)\, .
\end{eqnarray}

\begin{Lem}
\label{lemfghei0}
With the notation of Proposition \ref{profphf}. 
We assume that $\g u=\g u_\th$ is Heisenberg. Then\\
$a)$ The polynomial $F_{2d}$ is non-zero.\\
$b)$ There exists $x_0\in \g g_2$ and $y_0\in \g g_{-2}$ 
such that $[x_0,y_0]=h_0$.\\  
$c)$ Every automorphism $\ph\in H$ is graded i.e. 
$\ph(\g g_1)= \g g_1$ and 
$\ph(\g g_2)= \g g_2$.\\
$d)$ The bilinear form $G_2$  introduced in \eqref{eqng2xy} is a non-degenerate 
duality between $\g g_2$ and $\g g_{-2}$   
\end{Lem}

\begin{proof}[Proof of Lemma \ref{lemfghei0}] 
$a)$ The Lie algebra $\g g':=\g g_{-2}\oplus \g g'_0\oplus \g g_2$, with $\g g'_0=[\g g_{-2},\g g_2]$ is a simple
Lie algebra of real rank one. Therefore 
its horospherical subalgebra
$\g u':=\g g_2$ is reflexive commutative,
and Lemma \ref{lemfgcom}.a tells us that 
the polynomial $F_{2d}$ is non-zero.

$b)$ The polynomial $F_{2d}$ on $\g g_2$ is $L$-semi-invariant. 
The  existence of the $\g s\g l_2$-triple 
$(x_0,h_0,y_0)$ follows as 
in the proof of Lemma \ref{lemfgcom}.b. 

$c)$ Since $\g g_2$ is the center of $\g u$,
one has $\ph(\g g_2)= \g g_2$.
Moreover, since $\g u$ is Heisenberg, the group $L$ acts transitively on $\g g_2\smallsetminus \{0\}$, and hence
one has the equality $\g g_1=\{X\in \g u\mid F_{2d}(X)=0\}$.
Therefore, one also has $\ph(\g g_1)= \g g_1$.

$d)$ This follows from Proposition \ref{profphf}
applied to the 
reflexive commutative horospherical Lie subalgebra $\g u'$ of $\g g'$ introduced in $a)$.
\end{proof}

\begin{proof}[Proof of Lemma \ref{lemfphf}
for $\g u$ Heisenberg]
We have already seen that $G_2$ is a non-degenerate
duality between $\g g_{-2}$ and $\g g_2$.

This bilinear form $G_2$ is non-zero on $\g g_{-1}\times\g g_1$. Indeed, 
for every simple root $\al$ in $\th$, there exists
a $\g s\g l_2$-triple $(x_\al,h_\al,y_\al)$
with 
$x_\al$ in $\g g_\al$ and 
$y_\al$ in $\g g_{-\al}$. 
Since $\be:=\widetilde\al -\al$ is a root, one has
$[\g g_{\widetilde\al},\g g_{-\al}]=\g g_\be$,
and the action of $h_\al$ on $\g g_{\widetilde\al}$
is scalar and non-zero.
Therefore, one has 
$$
G_2(x_\al,y_\al)={\rm tr}_{\g z}({\rm ad}h_\al)\neq 0.
$$
This bilinear form $G_2$ is $L$-invariant and the action of $L$ on $\g g_1$ is either irreducible
or every irreducible component contains 
one root space $\g g_\al$, with $\al\in \th$. 
Therefore the bilinear form $G_2$ is also a non-degenerate duality 
between $\g g_{-1}$ and $\g g_1$.
\end{proof}

Remember that the set $\th$ contains 
either one simple root $\al$ with $\-w_0\al=\al$ 
or two simple roots   $\al$ and $-w_0\al$
of the same length.

It remains to prove Proposition \ref{profphf}.a.
We  split the proof into two cases:\\
\ref{secpolhei1}: when  the roots $\al\in \th$ have the same length as the largest root $\widetilde \al$.\\ 
\ref{secpolhei2}: when  $\th$ contains a root $\al $ with is shorter than  $\widetilde \al$.  

\subsubsection{When $U$ is Heisenberg
and $\|\al\|=\|\widetilde\al\|$}
\label{secpolhei1}
This corresponds to all cases in Tables \ref{tabhorhei} and  \ref{tabhorrea} except $(3)$ and $(10)$.

\begin{Lem}
\label{lemfghei1}
With the notation of Proposition \ref{profphf}. 
We assume that $\g u=\g u_\th$ is Heisenberg and $\|\al\|=\|\widetilde\al\|$ for all $\al\in \th$. Then\\
$a)$ There exists $x'_0\in \g g_1$ and $y'_0\in \g g_{-1}$ 
such that $h'_0:=[x'_0,y'_0] =2\, h_0$.\\  
$b)$ The polynomial $F_{4d}$ is non-zero.\\
$c)$ The Lie algebra $\g l$
is a maximal Lie subalgebra of the Lie algebra 
$\g d\g e\g r_{gr}(\g u)$ 
of graded derivations of $\g u$,
except in Case $(11)$ of Table \ref{tabhorrea} 
where $\g l=\g d\g e\g r_{gr}(\g u)$.   
\end{Lem}

\begin{Exa}
\label{exasl2sl3}
Let $\g g$ be the simple Lie algebra $\g g={\g s\g l}(3,\m R)$, 
let $\g u=\g g_1\oplus \g g_2$ be the upper triangular horospherical subalgebra  and $\g u^-=\g g_{-1}\oplus \g g_{-2}$ be the lower triangular horospherical subalgebra.
A $\g s\g l_2$-triple $(x_0, h_0,y_0 )$ 
satisfying Lemma  \ref{lemfghei0}.b is 
\begin{equation*}
x_0=\mbox{\scriptsize 
$\left(\!
\begin{array}{ccc} 0 &0&1   \\
  0 &0&0\\
  0&0&0  
\end{array}\!
\right)$}\in \g g_2
\, ,\;
h_0=\mbox{\scriptsize
$\left(\!
\begin{array}{ccc} 1 &0&0   \\
  0 &0&0\\
  0&0&\!-1  
\end{array}\!
\right)$}\in \g g_0
\, ,\;
y_0=\mbox{\scriptsize
$\left(\!
\begin{array}{ccc} 0 &0&0   \\
  0 &0&0\\
  1&0&0  
\end{array}\!
\right)$}\in \g g_{-2}\, .
\end{equation*}
The corresponding grading \eqref{eqng2g0g2} of $\g g$ is given by this element $h_0$. 
A $\g s\g l_2$-triple $(x'_0, h'_0,y'_0 )$ 
satisfying Lemma  \ref{lemfghei1}.a is
\begin{equation*}
x'_0=\mbox{\scriptsize 
$\left(\!
\begin{array}{ccc} 0 &1&0   \\
  0 &0&1\\
  0&0&0  
\end{array}\!
\right)$}\in \g g_1
\, ,\;
h'_0=\mbox{\scriptsize 
$\left(\!
\begin{array}{ccc} 2 &0&0   \\
  0 &0&0\\
  0&0&\!-2  
\end{array}\!
\right)$}\in \g g_0
\, ,\;
y'_0=\mbox{\scriptsize 
$\left(\!
\begin{array}{ccc} 0 &0&0   \\
  1 &0&0\\
  0&1&0  
\end{array}\!
\right)$}\in \g g_{-1}\, .
\end{equation*}
\end{Exa}

\begin{proof}[Proof of Lemma \ref{lemfghei1}] 
Since the subset $\th\subset \Pi$
is given by \eqref{eqnthhei}, the difference 
$\be:=\widetilde\al-\al$ is a root.
Since the roots $\widetilde \al$ and $\al$ have
the same length, the subset 
$\Si':=\{\pm\al,\pm\be,\pm\widetilde\al\} \subset \Sigma$
is a root system of type $A_2$.
Let $\g g'\subset \g g$ be the graded Lie subalgebra spanned by the root spaces
$\g g_{\pm\al}$, $\g g_{\pm\be}$ and $\g g_{\pm\widetilde\al}$.
It is a simple real Lie algebra with root system $\Sigma'$.
Note that by construction $\g g'_{\pm 2}=\g g_{\pm 2}$

By \cite[Theorem 7.2]{BorelTits65}, there exists a split 
simple algebraic subalgebra $\g g''\subset \g g'$
with the same real rank and same root system as $\g g'$. 
Therefore, this subalgebra $\g g''$ is a graded 
subalgebra of $\g g$ which is isomorphic to $\g s\g l(3,\m R)$
and the intersections $\g u'':=\g u\cap\g g''$ and
${\g u''}^-:=\g u^-\cap\g g''$ are the upper and lower triangular horospherical
subalgebras of  $\g s\g l(3,\m R)$ given in Example \ref{exasl2sl3}.

The first $\g s\g l_2$-triple $(x_0, h_0,y_0 )$ of $\g g''$
in Example \ref{exasl2sl3} is also the one given by  Lemma \ref{lemfghei0}.b,
and therefore $h_0$ is the element of $\g g_0$ inducing the graduation \eqref{eqng2g0g2} on $\g g$.
The second $\g s\g l_2$-triple $(x'_0, h'_0,y'_0 )$ of $\g g''$
in Example \ref{exasl2sl3} is the one we are looking for
since one has $h'_0=2h_0$.

b) We   use this  $\g s\g l_2$-triple $(x'_0, h'_0,y'_0 )$.
The centers
$\g z$ and $\g z^-$ are respectively the eigenspaces of 
${\rm ad}h'_0$ for the eigenvalues $4$ and $-4$.
Therefore the theory of $\g s\g l_2$-modules tells us that
the map $({\rm ad}x'_0)^4$ induces a bijection from $\g z^-$ to $\g z$, and 
$F_{4d}(x'_0)={\rm det}_\g z(
\tfrac{1}{24}({\rm ad}x'_0)^4w_0)$ is non-zero.

$c)$ Recall that the Heisenberg horospherical subalgebra $\g u$ is graded as 
$\g u=\g v\oplus\g z$. Therefore the group ${\rm Aut}_{gr}(\g u)$
can be seen as a subgroup of ${\rm GL}(\g v)$.
The Lie bracket on $\g u$ gives an antisymmetric bilinear map
$\om:\g v\times\g v\ra\g z$.
\vs

{\bf First Case} When $|\th|=1$.\\
If  $\g g$ is absolutely simple we replace $\g g$ by $\g g_\m C$. Hence, without loss of generality,
we can assume that $\g g$ is a complex simple Lie algebra.
In this case, $\g v$ is a complex vector space and $\om$
is a complex symplectic form on $\g v$.
The representation of $L$
in $\g v$ preserves the complex structure on $\g v$ and preserves up to scalar the symplectic form $\om$. This representation  is irreducible.
By Lemma \ref{lemauthei}, every graded automorphism $\ph$ of the real Lie algebra $\g u$ is complex linear or skew-linear.
The group of graded automorphisms of $\g u$ is up to finite index
${\rm Aut}_{gr}(\g u)\simeq {\rm Sp}(\g v,\om)\times \m C^*$.

The maximality follows from a direct analysis of each case 
from $(2)$ to $(8)$ occuring in Table \ref{tabhorhei}.
This straightforward analysis can again be seen as a special case of
Dynkin's classification 
in \cite{Dynkin00}. 
Indeed, since none of the examples of Table \ref{tabhorhei} belong to 
Dynkin's list of exceptions   \cite[Table 7 p.236]{OnishchikVinverg94},
we deduce that $\g l_{\m C}$ is a maximal subalgebra of $\g s\g p(\g v_\m C)\oplus \m C$.
\vs

{\bf Second Case} When $|\th|=2$.\\
This corresponds to Cases $(1)$, $(9)$ and $(11)$ in Tables \ref{tabhorhei} and \ref{tabhorrea}.

- Case $(1)$.
One has $\g g = \g s\g l(\m C^{n+2})$ or $\g g_\m C=\g s\g l(\m C^{n+2})$. Using again Lemma \ref{lemauthei}, one sees that the  Levi Lie algebra 
$\g l$ (resp. $\g l_\m C $) is isomorphic to $\g g\g l(\m C^n)\oplus \m C$ and hence is a maximal Lie subalgebra in 
$\g d\g e\g r_{gr}(\g u)$ (resp. $\g d\g e\g r_{gr}(\g u_\m C)$), itself isomorphic to $\g s\g p(\m C^{2n})\oplus\m C$.

- Case $(9)$.
One has 
$\g g_{\m C}=\g s\g l(\m C^{2n+4})$,
$\g v_{\m C}=\m C^2\otimes {\m C^n}^* \oplus 
{\m C^n}\otimes {\m C^2}^*$ 
and $\g z_{\m C}={\rm End}(\m C^2)$. Therefore the Levi Lie algebra $\g l_\m C\simeq \g g\g l(\m C^2)\oplus\g s\g l(\m C^{n})\oplus \g g\g l(\m C^2)$
is maximal in 
$\g d\g e\g r_{gr}(\g u_\m C)\simeq 
\g s\g l(\m C^2)\oplus\g s\g l(\m C^2)\oplus\g s\g p(\m C^{2n})\oplus\m C^2$.

- Case $(11)$.
One has $\g g_{\m C}=\g e_6$, $\g v_{\m C}=\m C^8\oplus\m C^8$ 
and $\g z_{\m C}=\m C^8$. The Levi Lie algebra is   
$\g l_\m C \simeq \g s\g o(\m C^{8})\oplus\m C^2$.
Here the three spaces $\m C^{8}$ are the three irreducible
$8$-dimensional representations  of $\g s\g o(\m C^{8})$, 
the standard one and the two half-spin representations. 
Therefore $\g l_\m C$
 is equal to
$\g d\g e\g r_{gr}(\g u_\m C)\simeq 
\g s\g o(\m C^{8})\oplus\m C^2$.
\end{proof}

In the previous proof, we have used the following.

\begin{Lem}
\label{lemauthei}
Let $\g u$ be a complex Heisenberg Lie algebra.
Every automorphism $\ph$ of  the real Lie algebra $\g u$
is either complex linear or skew-linear.
\end{Lem}

\begin{proof}[Proof of Lemma \ref{lemauthei}]
One checks it by direct computation. 
\end{proof}

\begin{proof}[Proof of Proposition \ref{profphf}.a
for $\g u$ Heisenberg, and 
$\|\al\|=\|\widetilde\al\|$ ] \mbox{ }\\
By Lemma \ref{lemfghei1}.a, the   function $F_{4d}$ on $\g v$ is $L$-semi-invariant and non-constant. 
When we are not in Case $(11)$ of Table \ref{tabhorrea}, the derived group of the group ${\rm Aut}_{gr}(\g u)$
has an open orbit in $\g v$, therefore this function $F_{4d}$ 
is not ${\rm Aut}_{gr}(\g u)$-semi\-invariant.
Therefore, by   Lemma \ref{lemfghei1}.c,
the Lie algebra $\g h$ of the group $H$ is equal to $\g l$ in all   cases.
\end{proof}

\subsubsection{When $U$ is Heisenberg
and $\|\al\|<\|\widetilde\al\|$}
\label{secpolhei2}

This corresponds to Case $(3)$ and $(10)$ 
in Tables \ref{tabhorhei}
and   \ref{tabhorrea}.

\begin{Lem}
\label{lemfghei2}
With the notation of Proposition \ref{profphf}. 
We assume that $\g u=\g u_\th$ is Heisenberg and $\|\al\|<\|\widetilde\al\|$. Then
the Lie algebra $\g l$
coincides with the algebra 
$\g d\g e\g r_{gr}(\g u)$ 
of graded derivations of $\g u$  
\end{Lem}

\begin{proof}[Proof of Lemma \ref{lemfghei2}]

- Case $(3)$. One has $\g g=\g s\g p(\m C^{2n+2})$ or 
$\g g_\m C=\g s\g p(\m C^{2n+2})$.  
Using again Lemma \ref{lemauthei}, one sees that the  Levi Lie algebra 
$\g l$ (resp. $\g l_\m C$) is isomorphic to $\g s\g p(\m C^{2n})\oplus \m C$ and equal to 
$\g d\g e\g r_{gr}(\g u)$ (resp. $\g d\g e\g r_{gr}(\g u_\m C)$).

- Case $(10)$.
One has 
$\g g_{\m C}=\g s\g p(\m C^{4n+8})$,
$\g v_{\m C}=\m C^2\otimes \m C^{2n} $, 
$\g z_{\m C}=S^2(\m C^2)$, 
and the   Levi Lie algebra  
$\g l_\m C\simeq \g g\g l(\m C^{2})\oplus\g s\g p(\m C^{2n})$ is equal to $\g d\g e\g r_{gr}(\g u_\m C)$.
\end{proof}

\begin{proof}[Proof of Proposition \ref{profphf}.a
for $\g u$ Heisenberg and 
$\|\al\|<\|\widetilde\al\|$ ]\mbox{ }\\
By Lemma \ref{lemfghei0}.c, 
the automorphisms $\ph\in H$ are graded automorphisms of $\g u$.
By Lemma \ref{lemfghei2}, the group $L$ has finite index in the group
${\rm Aut}_{gr}(\g u)$. Therefore the group $L$ also has finite index in $H$.
\end{proof}

%9

{\small
\noindent
Y. \textsc{Benoist}:
Universit\'e Paris-Sud, Orsay, France\newline
e-mail: \texttt{yves.benoist@u-psud.fr}}

\medskip
{\small
\noindent
S. \textsc{Miquel}:
Universit\'e Paris-Sud, Orsay, France\newline
e-mail: \texttt{sebastien.miquel@u-psud.fr}}

\end{document}